\documentclass[12pt]{article}

\usepackage{amsmath}
\usepackage{amssymb,latexsym}
\usepackage{amsthm}
\usepackage{enumerate}
\usepackage[hang,flushmargin]{footmisc}
\usepackage[margin=0.35in,format=hang,font=small,labelfont=bf]{caption}
\usepackage{graphicx,tikz}
\usepackage[colorlinks=true,linkcolor=blue,citecolor=blue]{hyperref}
\usepackage[greek,english]{babel}

\makeatletter
\newtheorem*{rep@theorem}{\rep@title}
\newcommand{\newreptheorem}[2]{
\newenvironment{rep#1}[1]{
\def\rep@title{#2~\ref{##1}}
\begin{rep@theorem}}
{\end{rep@theorem}}}
\makeatother

\newtheorem{thm}{Theorem}[section]
\newreptheorem{thm}{Theorem}
\newtheorem{cor}[thm]{Corollary}
\newreptheorem{cor}{Corollary}
\newtheorem{lemma}[thm]{Lemma}
\newreptheorem{lemma}{Lemma}

\newtheorem{obs}[thm]{Observation}

\newcommand{\+}{\hspace{0.1 em}}
\newcommand{\floor}[1]{\left\lfloor #1 \right\rfloor}
\newcommand{\ceil}[1]{\left\lceil #1 \right\rceil}
\newcommand{\half}{\frac{1}{2}}

\newcommand{\limkinfty}{\lim\limits_{k\rightarrow\infty}}
\newcommand{\tr}{\mathrm{tr}}
\newcommand{\dg}{d}

\newcommand{\CCC}{\mathcal{C}}

\newcommand{\SSS}{\mathcal{S}}
\newcommand{\TTT}{\mathcal{T}}

\newcommand{\WWW}{\mathcal{W}}

\newcommand{\WB}{W^{{}^{\text{\textsf{\textbf{B}}}}}}
\newcommand{\WWWB}{\WWW^{{}^{\text{\textsf{\textbf{B}}}}}}

\newcommand{\WWWE}{\WWW^{{}^{\text{\textsf{\textbf{E}}}}}}

\newcommand{\Hzero}{H_0}
\newcommand{\vis}[2]{\Psi(#1,#2)}

\newcommand{\xx}{\mbox{$\ast$-$\ast$}}
\newcommand{\ox}{\mbox{1-$\ast$}}
\newcommand{\xo}{\mbox{$\ast$-1}}
\newcommand{\oo}{\mbox{1-1}}
\newcommand{\xinit}{\mbox{$\ast$-initial}}

\newcommand{\Grid}{\mathrm{Grid}}

\newcommand{\Gridhash}{\Grid^\#}
\newcommand{\Geom}{\mathrm{Geom}}
\newcommand{\gr}{\mathrm{gr}}
\newcommand{\grup}{\overline{\gr}}
\newcommand{\grlow}{\underline{\gr}}
\newcommand{\av}{\mathrm{Av}}

\newenvironment{gridmx}[1][{}] {\Grid^{#1}\!\left(\!\begin{smallmatrix}} {\end{smallmatrix}\!\right)}

\newcommand{\gclass}[4][0.25]  
{
  \begin{tikzpicture}[scale=#1]
    \draw[very thin] (0,0) grid (#2,#3);
    #4
  \end{tikzpicture}
}

\newcommand{\gcrow}[2]  
{
  \foreach \d [count=\i,evaluate=\d as \shade using 100*\d*\d] in {#2}
    \draw[very thick,color=black!\shade!white] (\i-.82,#1+.5-.32*\d)--(\i-.18,#1+.5+.32*\d);
}

\newcommand{\gxrow}[2]  
{
  \foreach \d [count=\i,evaluate=\d as \shade using 100*\d*\d] in {#2}
  {
    \draw[thick,color=black!\shade!white] (\i-.82,#1+.18)--(\i-.18,#1+.82);
    \draw[thick,color=black!\shade!white] (\i-.82,#1+.82)--(\i-.18,#1+.18);
  }
}

\newcommand{\gctwo}[3]   {\raisebox{ -4.00pt}{\gclass{#1}{2}
                            {\gcrow{1}{#2}\gcrow{0}{#3}}}}

\newcommand{\gcfive}[6]  {\raisebox{-15.00pt}{\gclass{#1}{5}
                            {\gcrow{4}{#2}\gcrow{3}{#3}\gcrow{2}{#4}\gcrow{1}{#5}\gcrow{0}{#6}}}}

\newcommand{\gxone}[2]   {\raisebox{ -0.33pt}{\gclass{#1}{1}
                            {\gxrow{0}{#2}}}}
\newcommand{\gxtwo}[3]   {\raisebox{ -4.00pt}{\gclass{#1}{2}
                            {\gxrow{1}{#2}\gxrow{0}{#3}}}}
\newcommand{\gxthree}[4] {\raisebox{ -7.67pt}{\gclass{#1}{3}
                            {\gxrow{2}{#2}\gxrow{1}{#3}\gxrow{0}{#4}}}}
\newcommand{\gxfour}[5]  {\raisebox{-11.33pt}{\gclass{#1}{4}
                            {\gxrow{3}{#2}\gxrow{2}{#3}\gxrow{1}{#4}\gxrow{0}{#5}}}}

\newcommand{\plotperm}[2]  
{
  \foreach \y [count=\x] in {#2} \fill[radius=0.275] (\x,\y) circle;
  \draw[thick] (.5,.5) rectangle (#1.5,#1.5);
}

\newcommand{\plotpermgrid}[2]  
{
  \foreach \x in {1,...,#1} \draw[very thin] (\x,.5)--(\x,#1.5) (.5,\x)--(#1.5,\x);
  \foreach \y [count=\x] in {#2} \fill[radius=0.275] (\x,\y) circle;
  \draw[thick] (.5,.5) rectangle (#1.5,#1.5);
}

\title{
Growth rates of permutation grid classes, \\
tours on graphs, and the spectral radius
}

\author{David Bevan\\
\footnotesize Department of Mathematics and Statistics\\[-5pt]
\footnotesize The Open University\\[-5pt]
\footnotesize Milton Keynes, England\\[-5pt]
\footnotesize \texttt{David.Bevan@open.ac.uk}}

\date{}

\begin{document}
\maketitle

\let\thefootnote\relax\footnotetext
{2010 Mathematics Subject Classification: 05A05, 
05A16, 05C50. 
}

\begin{abstract}
\noindent
Monotone grid classes of permutations have proven very effective in helping to determine structural and enumerative properties of classical permutation pattern classes.
Associated with grid class $\Grid(M)$ is
a graph, $G(M)$,
known as its ``row-column''
graph.
We
prove that the
exponential
growth rate of $\Grid(M)$ is equal to the square of the spectral radius of $G(M)$.
Consequently,
we
utilize
spectral graph
theoretic results
to
characterise all slowly growing grid classes
and
to
show
that for every
$\gamma\geqslant2+\sqrt{5}$
there is a grid class with growth rate arbitrarily close to $\gamma$.
To prove our main result, we 
establish bounds on the size of certain families of tours on graphs.
In the process,
we
prove
that
the family of
tours of even length on
a connected graph
grows at the same rate as
the family of
``balanced''
tours on
the graph
(in which
the number of times an edge is traversed in one direction is the same as the number of times it is traversed in the other direction).
\end{abstract}

\section{Introduction}
We consider a permutation to be simply an arrangement of the numbers $1,2,\ldots,k$ for some positive $k$.
We use $|\sigma|$ to denote the length of permutation $\sigma$.
A permutation $\tau$ is said to be \textbf{contained} in, or to be a \textbf{subpermutation} of, another permutation $\sigma$ if $\sigma$ has a subsequence whose terms have the same relative ordering as
$\tau$.
It can be
helpful to consider permutations graphically, and
from the graphical perspective,
$\sigma$ contains $\tau$ if the plot of $\tau$ results from erasing some points from the plot of $\sigma$ and then ``shrinking'' the axes appropriately.
If $\sigma$ does not contains $\tau$, we say that $\sigma$ \textbf{avoids} $\tau$.
For example, 31567482 contains 1324 (see Figure~\ref{figPermutation}) but avoids 1243.
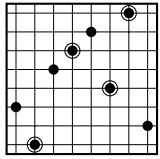
\begin{figure}[ht]
  $$
  \begin{tikzpicture}[scale=0.25]
    \plotpermgrid{8}{3,1,5,6,7,4,8,2}
    \draw [thin] (2,1) circle [radius=0.4];
    \draw [thin] (4,6) circle [radius=0.4];
    \draw [thin] (6,4) circle [radius=0.4];
    \draw [thin] (7,8) circle [radius=0.4];
  \end{tikzpicture}
  $$
  \caption{The plot of permutation 31567482, with a 1324 subpermutation marked}\label{figPermutation}
\end{figure}

Clearly, the containment relation is a partial order on the set of all permutations. A classical \textbf{permutation class} (or ``pattern class'') is a set of permutations closed downwards (a \emph{down-set}) in this partial order. From a graphical perspective, this means that erasing points from
the plot of
a permutation in a permutation class $\CCC$ always results in
the plot of
another permutation in $\CCC$
when the axes are rescaled appropriately.

Given a permutation class $\CCC$, we denote by $\CCC_k=\{\sigma\in\CCC:|\sigma|=k\}$ the set of permutations in $\CCC$ of length $k$.
The (ordinary) generating function of $\CCC$ is thus $\sum_{k\in\mathbb{N}}|\CCC_k|z^k=\sum_{\sigma\in\CCC}z^{|\sigma|}$.
It is common to define a permutation class $\CCC$ ``negatively'' by stating the minimal set of permutations $B$ that do \emph{not} occur in the class. In this case, we write $\CCC=\av(B)$ (where Av signifies ``avoids''). $B$ is called the \textbf{basis} of $\CCC$.
The basis of a permutation class is an \emph{antichain} (a set of pairwise incomparable elements)
and may be infinite.

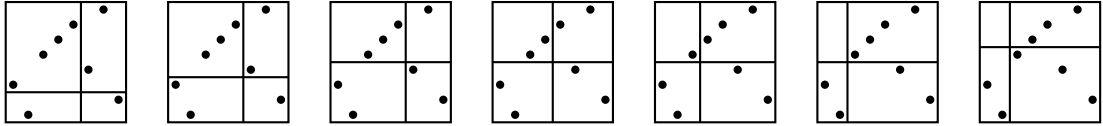
\begin{figure}[ht]
  $$
  \begin{tikzpicture}[scale=0.20]
    \plotperm{8}{3,1,5,6,7,4,8,2}
      \draw[thick] (5.5,.5) -- (5.5,8.5);
      \draw[thick] (.5,2.5) -- (8.5,2.5);
  \end{tikzpicture}
  \quad\;
  \begin{tikzpicture}[scale=0.20]
    \plotperm{8}{3,1,5,6,7,4,8,2}
      \draw[thick] (5.5,.5) -- (5.5,8.5);
      \draw[thick] (.5,3.5) -- (8.5,3.5);
  \end{tikzpicture}
  \quad\;
  \begin{tikzpicture}[scale=0.20]
    \plotperm{8}{3,1,5,6,7,4,8,2}
      \draw[thick] (5.5,.5) -- (5.5,8.5);
      \draw[thick] (.5,4.5) -- (8.5,4.5);
  \end{tikzpicture}
  \quad\;
  \begin{tikzpicture}[scale=0.20]
    \plotperm{8}{3,1,5,6,7,4,8,2}
      \draw[thick] (4.5,.5) -- (4.5,8.5);
      \draw[thick] (.5,4.5) -- (8.5,4.5);
  \end{tikzpicture}
  \quad\;
  \begin{tikzpicture}[scale=0.20]
    \plotperm{8}{3,1,5,6,7,4,8,2}
      \draw[thick] (3.5,.5) -- (3.5,8.5);
      \draw[thick] (.5,4.5) -- (8.5,4.5);
  \end{tikzpicture}
  \quad\;
  \begin{tikzpicture}[scale=0.20]
    \plotperm{8}{3,1,5,6,7,4,8,2}
      \draw[thick] (2.5,.5) -- (2.5,8.5);
      \draw[thick] (.5,4.5) -- (8.5,4.5);
  \end{tikzpicture}
  \quad\;
  \begin{tikzpicture}[scale=0.20]
    \plotperm{8}{3,1,5,6,7,4,8,2}
      \draw[thick] (2.5,.5) -- (2.5,8.5);
      \draw[thick] (.5,5.5) -- (8.5,5.5);
  \end{tikzpicture}
  \vspace{-6pt}
  $$
  \caption{The seven griddings of permutation 31567482 in
  \protect\gctwo{2}{1,1}{-1,-1}
  }\label{figGriddings}
\end{figure}
The monotone \textbf{grid class} $\Grid(M)$
is a permutation class defined by a matrix $M$, all of whose entries are in $\{0,1,-1\}$, which specifies the acceptable ``shape'' for plots of permutations in the class. Each entry of $M$ corresponds to a \textbf{cell} in a ``gridding'' of a permutation.
If the
entry is $1$, any points in the cell must form an increasing sequence; if the entry is $-1$, any points in the cell must form a decreasing sequence; if the entry is $0$, the cell must be empty.
For greater clarity, we denote grid classes by \textbf{cell diagrams} rather than by their matrices; for example, $\gctwo{3}{1,-1,0}{0,-1,1}=\begin{gridmx}1&-1&0\\0&-1&1\end{gridmx}$.
A permutation may have multiple possible griddings in a grid class (see Figure~\ref{figGriddings} for an example).

Recent years have seen much progress on understanding enumerative and structural properties of permutation classes.
The use of grid classes has proven particularly fruitful.
One focus of research has been the enumeration of permutation classes that have small bases (see~\cite{WikiEnumPermClasses}).
In this context
the first use of grid classes
(but not using that term)
was by Atkinson~\cite{Atkinson1999}, who
determined that
$$
\av(132,4321)
\;=\;
\gcfive{5}
  {0,0,0,0,1}
  {0,0,1,0,0}
  {1,0,0,0,0}
  {0,1,0,0,0}
  {0,0,0,1,0}
\cup
\gcfive{5}
  {0,0,0,0,1}
  {1,0,0,0,0}
  {0,0,0,1,0}
  {0,1,0,0,0}
  {0,0,1,0,0}
$$
and used the fact to enumerate this class of permutations.
More recently,
Albert, Atkinson and Brignall~\cite{AAB2011,AAB2012}
and
Albert, Atkinson and Vatter~\cite{AAV2012}
have demonstrated the practical uses of grid classes for permutation class enumeration
by determining the generating functions of
seven permutation classes whose bases consist of two permutations of
length four.

Another primary area of exploration has concerned the \emph{growth rates} of permutation classes.
Marcus and Tardos~\cite{MT2004} proved the conjecture of Stanley and Wilf that for any permutation class $\CCC$ except the class of all permutations there exists a constant $c$ such that $|\CCC_k|\leqslant c^k$ for all $k$. Thus, every permutation class with non-empty basis has finite lower and upper exponential growth rates defined, respectively, by
$$
\grlow(\CCC)\;=\;\liminf_{k\rightarrow\infty}|\CCC_k|^{1/k}
\quad
\text{and}
\quad
\grup(\CCC)\;=\;\limsup_{k\rightarrow\infty}|\CCC_k|^{1/k}.
$$
If the lower and upper growth rates coincide, then $\CCC$ has a \textbf{growth rate}, which we denote $\gr(\CCC)$.
(It is widely conjectured that \emph{every} permutation class has a growth rate.)
In~\cite{Vatter2011}, Vatter investigated the possible values of permutation class growth rates, and used generalised grid classes to characterize all the (countably many) permutation classes with growth rates below $\kappa\approx2.20557$. He also established that there are uncountably many permutation classes with growth rate $\kappa$, and in a separate paper~\cite{Vatter2010b}, showed that there are permutation classes having every growth rate above $\lambda\approx2.48188$.
(The behaviour between
$\kappa$ and $\lambda$ is the subject of ongoing research.)

Grid classes have also been a subject of investigation themselves. The first to be studied was
the class of \emph{skew-merged} permutations $\!\gctwo{2}{-1,1}{1,-1}\!$. Stankova~\cite{Stankova1994} and
K\'{e}dzy, Snevily and Wang~\cite{KSW1996} proved that this class is $\av(2143,3412)$, and Atkinson~\cite{Atkinson1998} determined its generating function. More recently,
Waton, in his PhD thesis~\cite{WatonThesis}, enumerated $\!\gctwo{2}{1,1}{1,0}\!$.
In addition to these enumerations,
some structural results have also been established.
Atkinson~\cite{Atkinson1999}
proved that
grid classes whose matrices have dimension $1\times m$
have a finite basis.
Waton~\cite{WatonThesis} proved the same for
$\!\gctwo{2}{1,1}{1,1}\!$,
a result
which has been extended by
Albert, Atkinson and Brignall~\cite{AAB2010}
to all $2\times 2$
grid classes.
(It is generally believed that \emph{all} grid classes have a finite basis, but this has not yet been proven; see~\cite{HV2006} Conjecture 2.3.)

Associated with each grid class is a bipartite graph known as its ``row-column'' graph,
which encapsulates certain 
structural information about the class.
(We present its definition later in Section~\ref{sectGridClasses}.)
Particularly of note,
Murphy and Vatter~\cite{MV2003}
have shown that
a grid class is \emph{partially well-ordered} (contains no infinite antichains) if and only if its row-column graph has no cycles.
Moreover, Albert, Atkinson, Bouvel, Ru\v{s}kuc and Vatter~\cite{AABRV2011} proved a result that implies that
if a grid class has an acyclic row-column graph
then the generating function of the class is a rational function (the ratio of two polynomials).

Our focus in this paper is on the
growth rates
of grid classes.
We prove the following theorem:
\begin{repthm}{thmGrowthRate}
The growth rate of a monotone grid class of permutations exists and is equal to the square of the spectral radius of its row-column graph.
\end{repthm}

The bulk of the work required to prove this theorem 
is concerned with carefully counting certain families of tours on graphs, in order to give bounds on their sizes.
In particular, we consider ``balanced'' tours, in which the number of times an edge is
traversed in one direction is the same as the number of times it is traversed in the other
direction.
As a consequence, we 
prove the following new result concerning tours on graphs:
\begin{repthm}{thmBalancedEqualsEven}
The growth rate of the family of balanced tours on a connected graph is the same as that of the family of all tours of even length on the graph.
\end{repthm}

As a consequence of Theorem~\ref{thmGrowthRate}, 
by using the machinery of spectral graph theory, we are able to deduce
a variety of supplementary results.
We give a characterisation of grid classes whose growth rates are no greater than
$\frac{9}{2}$ 
(in a similar fashion to Vatter's characterisation of ``small'' permutation classes in~\cite{Vatter2011}).
We also fully characterise all \emph{accumulation points} of grid class growth rates, the least of which occurs at 4.
Other results include:
\begin{repcor}{corAlgebraicInteger}
The growth rate of every monotone grid class is an algebraic integer.
\end{repcor}
\begin{repcor}{corCycle}
A monotone grid class whose row-column graph is a cycle has growth rate~4.
\end{repcor}
\begin{repcor}{corSmallGrowthRates}
If the growth rate of a monotone grid class is less than 4, it is equal to $4\cos^2\!\left(\frac{\pi}{k}\right)$ for some $k\geqslant 3$.
\end{repcor}
\begin{repcor}{corAccumulationPoints}
For every $\gamma \geqslant 2+\sqrt{5}$ there is a monotone grid class with growth rate arbitrarily close to $\gamma$.
\end{repcor}


The remainder of this paper is structured as follows:
In Section~\ref{sectTours},
we introduce the particular families of tours on graphs that we study and
present our results concerning these tours, culminating in the proof of Theorem~\ref{thmBalancedEqualsEven}.
This is followed, in Section~\ref{sectGridClasses}, by the application of these results to prove our grid class growth rate result, Theorem~\ref{thmGrowthRate}.
Section~\ref{sectConsequences}, contains a number of consequences of Theorem~\ref{thmGrowthRate} that follow from known spectral graph theoretic results.
We conclude with a few final remarks.

\section{Tours on graphs}\label{sectTours}
In this section, we investigate families of tours on graphs, parameterised by the number of times each edge is traversed.
We determine a lower bound on the size of families of ``balanced'' tours
and an upper bound on families of arbitrary tours.
Applying the upper bound to tours of even length gives us an expression compatible with the lower bound. 
Combining this with the fact that any balanced tour has even length enables us to
prove Theorem~\ref{thmBalancedEqualsEven}
which reveals that even-length tours and balanced tours grow at the same rate.
These bounds are subsequently used in Section~\ref{sectGridClasses} to relate
tours on graphs to permutation grid classes.

To establish the lower and upper bounds, we first enumerate tours on trees. We then present a way of associating tours on an arbitrary connected graph $G$ with tours on a
related 
``partial covering'' tree,
which we employ to determine bounds for families of tours on arbitrary graphs.
Let us begin by introducing the tours that we will be considering.

\subsection{Notation and definitions}
A \emph{walk}, of length $k$, on a graph
is a non-empty alternating sequence of
vertices and edges $v_0, e_1, v_1, e_2, v_2, \ldots, e_k, v_k$ in which the endvertices of $e_i$ are $v_{i-1}$ and $v_i$.
Neither the edges nor the vertices need be distinct.
We say that such a walk \textbf{traverses} edges $\{e_1,\ldots,e_k\}$ and \textbf{visits} vertices $\{v_1,\ldots,v_{k-1}\}$.
A \emph{tour} (or \emph{closed} walk) is a walk which starts and ends at the same vertex (i.e. $v_k=v_0$).
Our interest is restricted to tours.

In what follows, when considering a graph with $m$ edges, we denote its edges $e_1,e_2,\ldots,e_m$. In any particular context, we can choose
the
ordering of the edges
so as
to simplify our presentation.
We denote the
edges incident to a given vertex $v$
by
${e^v_1},{e^v_2},\ldots,e^v_{\dg(v)}$, where $\dg(v)$ is the \emph{degree} of $v$ (number of edges incident to $v$).
Again, we are free to choose
the order of the edges incident to a vertex
so as
to clarify our arguments.

\subsubsection*{Families of tours}
Our interest is in families of tours that are parameterised by the number of times each edge is traversed.
Given non-negative integers $h_1,h_2,\ldots,h_m$
and some vertex $u$ of a graph $G$, we use
$$\WWW_G((h_i);u)\;=\;\WWW_G(h_1,h_2,\ldots,h_m;u)$$ to denote the family of tours on $G$ which start and end at $u$ and traverse each edge $e_i$ exactly $h_i$ times. (We use $\WWW$ rather than $\TTT$ for families of tours to avoid confusion when considering tours on trees.) 

We use ${h^v_1},{h^v_2},\ldots,h^v_{\dg(v)}$ for the number of traversals of edges incident to a vertex $v$ in $\WWW_G((h_i);u)$.
So, if $v$ and $w$ are the endvertices of $e_i$, $h_i$ has two aliases $h^v_j$ and $h^w_{j'}$ for some $j$ and $j'$.

We use $W_G((h_i);u)=|\WWW_G((h_i);u)|$ to denote the number of these tours.

Note that for some values of $h_1,\ldots,h_m$, the family $\WWW_G((h_i);u)$ is empty.
In particular, if $E^+=\{e_i\in E(G):h_i>0\}$ is the set of edges visited by tours in the family,
and $G^+=G[E^+]$ is the subgraph of $G$ induced by these edges,
then
if $G^+$ is disconnected or does not contain $u$,
we have $\WWW_G((h_i);u)=\varnothing$.
A family of tours may also be empty for ``parity'' reasons; for example, if $T$ is a tree, then $\WWW_T((h_i);u)=\varnothing$ if any of the $h_i$ are odd. Our counting arguments must remain valid for these empty families.

Of particular interest to us are tours in which
the number of times an edge is traversed in one direction is the same as
the number of times it is traversed in the other direction.
We call such tours {\textbf{balanced}}.

Given non-negative integers $k_1,k_2,\ldots,k_m$
and some vertex $u$ of a graph $G$, we use
$$\WWWB_G((k_i);u)\;=\;\WWWB_G(k_1,k_2,\ldots,k_m;u)$$ to denote the family of balanced tours on $G$ which start and end at $u$, and traverse each edge $e_i$ exactly $k_i$ times \emph{in each direction}.
Note that we parameterise balanced tours by \emph{half} the number of traversals of each edge.

We use ${k^v_1},{k^v_2},\ldots,k^v_{\dg(v)}$ for the number of traversals in either direction of edges incident to a vertex $v$ in $\WWWB_G((k_i);u)$.
So, if $v$ and $w$ are the endvertices of $e_i$, $k_i$ has two aliases $k^v_j$ and $k^w_{j'}$ for some $j$ and $j'$

We use $\WB_G((k_i);u)=|\WWWB_G((k_i);u)|$ to denote the number of these balanced tours.

As with $\WWW_G((h_i);u)$, $\WWWB_G((k_i);u)$ may be empty.
Observe also that, since any tour on a forest is balanced, $\WWW_F((2k_i);u)=\WWWB_F((k_i);u)$ for any forest $F$ and $u\in V(F)$.
Moreover, for any graph $G$,
we have
$\WB_G((k_i);u)\leqslant W_G((2k_i);u)$, with equality if and only if the component of $G^+$ containing $u$, if present, is acyclic, where $G^+$ is the subgraph of $G$ induced by the edges that are actually traversed by tours
in the family.

\subsubsection*{Visits and excursions}
We use $\vis{G}{v}$ to denote the number of \textbf{visits} to $v$ of any tour on $G$ in some
family (specified by the context).
In practice, this notation is unambiguous because we only consider one family of tours on a particular graph at a time.
Observe that any tour in $\WWW_G((h_i);u)$ visits vertex $v\neq u$ exactly $\half({h^v_1}+{h^v_2}+\ldots+{h^v_{\dg(v)}})$ times, and
that for balanced tours in $\WWWB_G((k_i);u)$ we have
$\vis{G}{v}={k^v_1}+{k^v_2}+\ldots+{k^v_{\dg(v)}}$.

If $\vis{G}{v}$ is positive, then
separating the visits to $v$ are $\vis{G}{v}-1$ ``subtours''
starting and ending at $v$;
we refer to these subtours as \textbf{excursions} from $v$.

\subsubsection*{Multinomial coefficients}
In our calculations, we make considerable use of \emph{multinomial coefficients},
with their combinatorial interpretation,
for which we
use the standard notation
$$
\qquad\qquad\qquad\qquad
\binom{n}{k_1,k_2,\ldots,k_r}
\;=\; \frac{n!}{k_1!\+k_2!\+\ldots \+k_r!},
\qquad\quad
\text{where~}
\sum_{i=1}^r k_i =n,
$$
to denote
the number of ways of distributing $n$ distinguishable objects between $r$ (distinguishable) bins,
such that bin $i$ contains exactly $k_i$ objects ($1\leqslant i\leqslant r$).

We make repeated use of the fact that a multinomial coefficient can be decomposed into a product of binomial coefficients as follows:
$
\binom{n}{k_1,\ldots,k_r}
=
\binom{k_1}{k_1} \binom{k_1+k_2}{k_2} \ldots \binom{k_1+\ldots +k_r}{k_r}
.
$
We consider a multinomial coefficient that has one or more negative terms to be \emph{zero}.
This guarantees
that the monotonicity condition
$
\binom{n}{k_1,\ldots,k_r}
\leqslant
\binom{n+1}{k_1+1,\ldots,k_r}
$
holds
for
all possible sets of values.

\subsection{Tours on trees}\label{sectTrees}
We begin by establishing bounds on the size of families of tours on \emph{trees}.
As we noted above, all such tours are balanced.
We start with star graphs, giving an exact enumeration of any family:
\begin{lemma}\label{lemmaStar}
If $S_m$ is the star graph $K_{1,m}$ with central vertex $u$, then
$$
\WB_{S_m}((k_i);u)\;=\;
\binom{k_1+k_2+\ldots+k_m}{k_1,\,k_2,\,\ldots,\,k_m}
\;=\;
\binom{\vis{S_m}{u}}{k^u_1,\,k^u_2,\,\ldots,\,k^u_{\dg(u)}}.
$$
\end{lemma}
\begin{proof}
$\WWWB_{S_m}((k_i);u)$ consists of all possible interleavings of $k_i$ excursions from $u$ out-and-back along each $e_i$.
\end{proof}
It is possible to extend our exact enumeration to those families of balanced tours on trees in which every internal (non-leaf) vertex is visited at least once. These families are never empty.
\begin{lemma}\label{lemmaTree}
If $T$ is a tree,
$u\in V(T)$
and, for each $v\neq u$, $e^v_1$ is the edge incident to $v$ that is on the unique path between $u$ and $v$, and
if $k^v_1$ is positive for all internal vertices $v$ of $T$,
then
$$
\WB_T((k_i);u)\;=\;
\binom{\vis{T}{u}}{k^u_1,\,k^u_2,\,\ldots,\,k^u_{\dg(u)}}
\prod_{v\neq u}\binom{\vis{T}{v}-1}{{k^v_1}\!-\!1,\,{k^v_2},\,\ldots,\,{k^v_{\dg(v)}}}.
$$
\end{lemma}
\begin{proof}
We use induction on the number of internal vertices. Note that the multinomial coefficient for a leaf vertex simply contributes a factor of 1 to the product.
Lemma~\ref{lemmaStar} provides the base case.

Given a tree $T$ with $m$ edges $e_1,\ldots,e_m$,
and a leaf $v$ of $T$,
let $T'$ be the tree ``grown'' from $T$ by attaching $r$ new pendant edges $e_{m+1},\ldots,e_{m+r}$ to $v$.

If ${k^v_1}$ is positive,
since $v$ is a leaf, each tour in $\WWWB_T(k_1,\ldots,k_m;u)$ visits $v$ exactly ${k^v_1}$ times, with ${k^v_1}-1$ excursions from $v$ along ${e^v_1}$ separating these visits.
Any such tour can be extended to a tour in $\WWWB_{T'}(k_1,\ldots,k_{m+r};u)$
by arbitrarily interleaving $k_{m+i}$ new excursions out-and-back along each new pendant edge $e_{m+i}$
($i=1,\ldots,r$)
with the existing ${k^v_1}-1$ excursions from $v$ along ${e^v_1}$.
\end{proof}
This exact enumeration can be used to generate 
the following general bounds on the number of tours on trees:
\begin{cor}\label{corTreeBounds}
If $T$ is a tree, then for any vertex
$u\in V(T)$, $\WB_T((k_i);u)$ satisfies the following bounds:
$$
\prod_{v\in V(T)}\binom{\vis{T}{v}-\dg(v)}{{k^v_1}\!-\!1,\,{k^v_2}\!-\!1,\,\ldots,\,{k^v_{\dg(v)}}\!-\!1}
\;\leqslant\;
\WB_T((k_i);u)
\;\leqslant\;
\prod_{v\in V(T)}\binom{\vis{T}{v}}{{k^v_1},\,{k^v_2},\,\ldots,\,{k^v_{\dg(v)}}}.
$$
\end{cor}
\begin{proof}
If all the $k_i$ are positive, then this follows directly from Lemma~\ref{lemmaTree}.

If one or more of the $k_i$ is zero, then the lower bound is trivially true, because one of the multinomial coefficients is zero. The upper bound
also holds trivially
if there are no tours in the family. Otherwise,
let $T^+$ be the subtree of $T$ induced by the vertices actually visited by tours in $\WWWB_T((k_i);u)$. Then $\WB_T((k_i);u)=\WB_{T^+}((k_i);u)$. But we know that
$$
\WB_{T^+}((k_i);u)
\;\leqslant\;
\prod_{v\in V(T^+)}\binom{\vis{T^+}{v}}{{k^v_1},\,{k^v_2},\,\ldots,\,{k^v_{\dg(v)}}}
\;=\;
\prod_{v\in V(T)}\binom{\vis{T}{v}}{{k^v_1},\,{k^v_2},\,\ldots,\,{k^v_{\dg(v)}}}
$$
as a result of Lemma~\ref{lemmaTree} and the fact that $k^v_i=0$ for all
edges $e^v_i$ incident to
unvisited vertices
$v\in V(T)\setminus V(T^+)$.
\end{proof}

\subsection{Treeification} 
In order to establish the lower and upper bounds for tours on arbitrary connected graphs, we relate tours on a connected graph $G$ to (balanced) tours on a related
tree which we call a \textbf{treeification} of $G$.
The process of treeification
consists of repeatedly breaking cycles until
the resulting graph is acyclic.
This creates a sequence of graphs $G=G_0,G_1,\ldots,G_t=T$ 
where
$T$ is a tree.
We call this sequence a \textbf{treeification sequence}.

Formally, we define a treeification of a
connected
graph to be the result of 
the following
(nondeterministic)
process that transforms a connected graph
into a
tree
with
the same number of
edges. 

\newpage 
\begin{samepage}
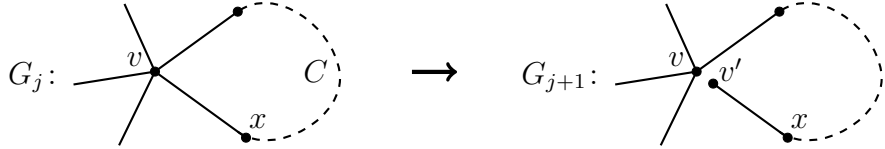
\begin{figure}[ht]
  $$
    \raisebox{.4in}{$G_j\!:\;\:$}
    \begin{tikzpicture}[scale=1.36,rotate=-36]
      \draw [thick] (1.1,0)--(0,0);
      \draw [thick] (0,0) -- (.31,.95);
      \draw [thick] (-.63,.36)--(0,0)--(-.57,-.57);
      \draw [thick] (0,0)--(.14,-.79);
      \draw [dashed,thick,] (1,0) to [out=0,in=-45] (1.46,1.06) node[left]{$C$} to [out=135,in=72] (.31,.95);
      \draw [fill] (0,0) circle [radius=0.044];
      \draw [fill] (1.1,0) circle [radius=0.044];
      \draw [fill] (.31,.95) circle [radius=0.044];
      \node at (-.2,0) {$v\:$};
      \node at (1.1,0.2) {$x$};
    \end{tikzpicture}
    \qquad
    \raisebox{.45in}
    {\begin{tikzpicture}\draw [->,ultra thick](0,0)--(.6,0);\end{tikzpicture}}
    \qquad
    \raisebox{.4in}{$G_{j+1}\!:\;\:$}
    \begin{tikzpicture}[scale=1.36,rotate=-36]
      \draw [thick] (1.1,0)--(.2,0);
      \draw [thick] (0,0) -- (.31,.95);
      \draw [thick] (-.63,.36)--(0,0)--(-.57,-.57);
      \draw [thick] (0,0)--(.14,-.79);
      \draw [dashed,thick,] (1,0) to [out=0,in=-45] (1.46,1.06) to [out=135,in=72] (.31,.95);
      \draw [fill] (0,0) circle [radius=0.044];
      \draw [fill] (.2,0) circle [radius=0.044];
      \draw [fill] (1.1,0) circle [radius=0.044];
      \draw [fill] (.31,.95) circle [radius=0.044];
      \node at (-.2,0) {$v\:$};
      \node at (.26,.21) {$v'$};
      \node at (1.1,0.2) {$x$};
    \end{tikzpicture}
    \vspace{-3pt}
  $$
  \caption{Splitting vertex $v$}\label{figVertexSplitting}
\end{figure}
To treeify a connected graph $G=G_0$,
we first give an (arbitrary) order to
its vertices. Then we 
apply the following
vertex-splitting operation in turn
to each $G_j$
to create $G_{j+1}$ ($j=0,1,\ldots$), 
until no cycles remain:
\vspace{-9pt}
\begin{enumerate}
  \itemsep0pt
  \item
  Let $v$ be the first vertex (in the ordering)
  that occurs in some cycle $C$ of $G_j$.
  \item Split vertex $v$ by doing the following (see Figure~\ref{figVertexSplitting}): \vspace{-3pt}
  \begin{enumerate}
  \itemsep0pt
  \item Delete an edge $xv$ from $E(C)$ (there are two choices for vertex $x$).
  \item Add a new vertex $v'$ (to the end of the vertex ordering).
  \item Add the pendant edge $xv'$ (making $v'$ a leaf).
  \end{enumerate}
\end{enumerate}
\end{samepage}

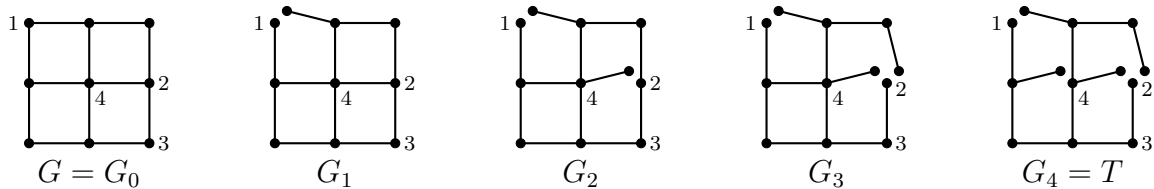
\begin{figure}[ht]
$$
  \begin{tikzpicture}[scale=0.8]
      \draw [thick] (0,2)--(0,0)--(2,0)--(2,1);
      \draw [thick] (1,0)--(1,2)--(2,2);
      \draw [thick] (1,2)--(0,2);
      \draw [thick] (1,1)--(2,1);
      \draw [thick] (2,2)--(2,1);
      \draw [thick] (0,1)--(1,1);
      \draw [fill] (0,0) circle [radius=0.075];
      \draw [fill] (0,1) circle [radius=0.075];
      \draw [fill] (0,2) circle [radius=0.075];
      \draw [fill] (1,0) circle [radius=0.075];
      \draw [fill] (1,1) circle [radius=0.075];
      \draw [fill] (1,2) circle [radius=0.075];
      \draw [fill] (2,0) circle [radius=0.075];
      \draw [fill] (2,1) circle [radius=0.075];
      \draw [fill] (2,2) circle [radius=0.075];
      \node at (1,-.5) {$G=G_0$};
      \node at (-.25,2) {${}_1$};
      \node at (2.25,1) {${}_2$};
      \node at (2.25,0) {${}_3$};
      \node at (1.2,.75) {${}_4$};
    \end{tikzpicture}
    \qquad
    \begin{tikzpicture}[scale=0.8]
      \draw [thick] (0,2)--(0,0)--(2,0)--(2,1);
      \draw [thick] (1,0)--(1,2)--(2,2);
      \draw [thick] (1,2)--(0.2,2.2);
      \draw [thick] (1,1)--(2,1);
      \draw [thick] (2,2)--(2,1);
      \draw [thick] (0,1)--(1,1);
      \draw [fill] (0,0) circle [radius=0.075];
      \draw [fill] (0,1) circle [radius=0.075];
      \draw [fill] (0,2) circle [radius=0.075];
      \draw [fill] (0.2,2.2) circle [radius=0.075];
      \draw [fill] (1,0) circle [radius=0.075];
      \draw [fill] (1,1) circle [radius=0.075];
      \draw [fill] (1,2) circle [radius=0.075];
      \draw [fill] (2,0) circle [radius=0.075];
      \draw [fill] (2,1) circle [radius=0.075];
      \draw [fill] (2,2) circle [radius=0.075];
      \node at (1,-.5) {$G_1$};
      \node at (-.25,2) {${}_1$};
      \node at (2.25,1) {${}_2$};
      \node at (2.25,0) {${}_3$};
      \node at (1.2,.75) {${}_4$};
    \end{tikzpicture}
    \qquad
    \begin{tikzpicture}[scale=0.8]
      \draw [thick] (0,2)--(0,0)--(2,0)--(2,1);
      \draw [thick] (1,0)--(1,2)--(2,2);
      \draw [thick] (1,2)--(0.2,2.2);
      \draw [thick] (1,1)--(1.8,1.2);
      \draw [thick] (2,2)--(2,1);
      \draw [thick] (0,1)--(1,1);
      \draw [fill] (0,0) circle [radius=0.075];
      \draw [fill] (0,1) circle [radius=0.075];
      \draw [fill] (0,2) circle [radius=0.075];
      \draw [fill] (0.2,2.2) circle [radius=0.075];
      \draw [fill] (1,0) circle [radius=0.075];
      \draw [fill] (1,1) circle [radius=0.075];
      \draw [fill] (1,2) circle [radius=0.075];
      \draw [fill] (2,0) circle [radius=0.075];
      \draw [fill] (2,1) circle [radius=0.075];
      \draw [fill] (1.8,1.2) circle [radius=0.075];
      \draw [fill] (2,2) circle [radius=0.075];
      \node at (1,-.5) {$G_2$};
      \node at (-.25,2) {${}_1$};
      \node at (2.25,1) {${}_2$};
      \node at (2.25,0) {${}_3$};
      \node at (1.2,.75) {${}_4$};
    \end{tikzpicture}
    \qquad
    \begin{tikzpicture}[scale=0.8]
      \draw [thick] (0,2)--(0,0)--(2,0)--(2,1);
      \draw [thick] (1,0)--(1,2)--(2,2);
      \draw [thick] (1,2)--(0.2,2.2);
      \draw [thick] (1,1)--(1.8,1.2);
      \draw [thick] (2,2)--(2.2,1.2);
      \draw [thick] (0,1)--(1,1);
      \draw [fill] (0,0) circle [radius=0.075];
      \draw [fill] (0,1) circle [radius=0.075];
      \draw [fill] (0,2) circle [radius=0.075];
      \draw [fill] (0.2,2.2) circle [radius=0.075];
      \draw [fill] (1,0) circle [radius=0.075];
      \draw [fill] (1,1) circle [radius=0.075];
      \draw [fill] (1,2) circle [radius=0.075];
      \draw [fill] (2,0) circle [radius=0.075];
      \draw [fill] (2,1) circle [radius=0.075];
      \draw [fill] (1.8,1.2) circle [radius=0.075];
      \draw [fill] (2.2,1.2) circle [radius=0.075];
      \draw [fill] (2,2) circle [radius=0.075];
      \node at (1,-.5) {$G_3$};
      \node at (-.25,2) {${}_1$};
      \node at (2.25,.9) {${}_2$};
      \node at (2.25,0) {${}_3$};
      \node at (1.2,.75) {${}_4$};
    \end{tikzpicture}
    \qquad
    \begin{tikzpicture}[scale=0.8]
      \draw [thick] (0,2)--(0,0)--(2,0)--(2,1);
      \draw [thick] (1,0)--(1,2)--(2,2);
      \draw [thick] (1,2)--(0.2,2.2);
      \draw [thick] (1,1)--(1.8,1.2);
      \draw [thick] (2,2)--(2.2,1.2);
      \draw [thick] (0,1)--(0.8,1.2);
      \draw [fill] (0,0) circle [radius=0.075];
      \draw [fill] (0,1) circle [radius=0.075];
      \draw [fill] (0,2) circle [radius=0.075];
      \draw [fill] (0.2,2.2) circle [radius=0.075];
      \draw [fill] (1,0) circle [radius=0.075];
      \draw [fill] (1,1) circle [radius=0.075];
      \draw [fill] (0.8,1.2) circle [radius=0.075];
      \draw [fill] (1,2) circle [radius=0.075];
      \draw [fill] (2,0) circle [radius=0.075];
      \draw [fill] (2,1) circle [radius=0.075];
      \draw [fill] (1.8,1.2) circle [radius=0.075];
      \draw [fill] (2.2,1.2) circle [radius=0.075];
      \draw [fill] (2,2) circle [radius=0.075];
      \node at (1,-.5) {$G_4=T$};
      \node at (-.25,2) {${}_1$};
      \node at (2.25,.9) {${}_2$};
      \node at (2.25,0) {${}_3$};
      \node at (1.2,.75) {${}_4$};
    \end{tikzpicture}
$$
\caption{A treeification sequence; numbers show the first few vertices in the ordering}\label{figTreeify}
\end{figure}

Note that if a vertex $v$ is split
multiple
times when treeifying a graph $G$, these
splits occur contiguously (because of the ordering placed on the vertices of $G$).
Thus, if $v$ is split
$r$ times, there is a contiguous subsequence $G_j,G_{j+1},\ldots,G_{j+r}$
of the treeification sequence
that corresponds to the splitting of $v$.
See Figure~\ref{figTreeify} for an example
of a treeification sequence.

There is a natural way to establish a relationship between tours on different graphs in a treeification sequence $G_0,...,G_t$.
The treeification process induces \emph{graph homomorphisms} (edge preserving maps) between the graphs in such a sequence.
For all $i<j$, there is a \emph{surjective} 
homomorphism from $G_j$ onto $G_i$.
This homomorphism is also \emph{locally injective} since it maps neighbourhoods of $G_j$ injectively into neighbourhoods of $G_i$.
A locally injective map such as this is also known as a \emph{partial cover}.
In particular, for each $j<t$, there is a
partial cover of $G_{j+1}$ onto $G_j$ that maps
the new pendant edge $xv'$ to the edge $xv$ that it replaces.
These homomorphisms
impart
a natural correspondence between families of tours on different graphs in the treeification sequence, which we will
employ
later
to determine our bounds.


Although the concept of treeification is a very natural one,
these particular ``partial covering trees'' do not appear to have been studied before;
their
only previous use seems to be by
Yarkony, Fowlkes and Ihler
to address a problem in
computer vision \cite{YFI2010}.
For a general introduction to graph homomorphisms, see
see the monograph by Hell and Ne\v{s}e\-t\v{r}il~\cite{HN2004}.
For more on partial maps and other locally constrained graph homomorphisms,
see the survey article by Fiala and Kratochv\'{\i}l~\cite{FK2008}.

If we have a treeification sequence
$G=G_0,\ldots,G_t=T$ for a connected graph $G$,
we can use a three-stage process to establish a lower or upper bound for a family of tours on $G$.
In the first stage (``splitting once''), we relate the number of tours in the family on $G_j$ ($j<t$) to the number of tours in a related family on $G_{j+1}$.
In the second stage (``fully splitting one vertex''), for a vertex $v$, we consider the subsequence $G_j,\ldots,G_{j+r}$
that corresponds to the splitting of $v$
and, iterating
the inequality from
the first stage, relate the number of tours on $G_j$ to the number of tours on $G_{j+r}$.
Finally (``fully splitting all vertices''), iterating
the inequality from
the second stage, we relate the number of tours on $G=G_0$ to the number of tours on $G_t=T$, and employ
the bounds on tours on $T$ from Corollary~\ref{corTreeBounds} to determine the bound for the family of tours on $G$.

In Subsection~\ref{sectLowerBound}, we use this three-stage process to produce a lower bound on
$\WB_G((k_i);u)$.
Then, in Subsection~\ref{sectUpperBound}, we use the same three-stage process to establish an upper bound on
$W_G((h_i);u)$.

\subsection{The lower bound}\label{sectLowerBound}
Our lower bound is
on the number of \emph{balanced} tours. We only
consider the families in which \emph{every} edge is traversed at least once in each direction. On a connected graph, these families are never empty.
\begin{lemma}
\label{lemmaLowerBound}
If $G$ is a connected graph with $m$ edges and $k_1,\ldots,k_m$ are all positive, then for any vertex $u\in V(G)$,
$$
\WB_G(k_1,k_2,\ldots,k_m;u)
\;\geqslant\;
\prod_{v\in V(G)}\binom
{{k^v_1}+{k^v_2}+\ldots+{k^v_{\dg(v)}}-\dg(v)}
{{k^v_1}\!-\!1,\,{k^v_2}\!-\!1,\,\ldots,\,{k^v_{\dg(v)}}\!-\!1}.
$$
\end{lemma}
This lower bound does not hold in general for a \emph{disconnected} graph since there are no tours possible if
there is any positive $k_i$
in a component not containing $u$.

\begin{proof}
Let $T$ be some treeification of $G$ with treeification sequence $G=G_0,\ldots,G_t=T$ in which
vertex $u$ is never split.
(This is possible by positioning $u$ last in the ordering on the vertices.)

By exhibiting a surjection from $\WWWB_G((k_i);u)$ onto $\WWWB_T((k_i);u)$ that is consistent with the homomorphism from $T$ onto $G$ induced by the treeification process, we
determine an inequality relating
the number of tours
in the two families.

\subsubsection*{I. Splitting once}
Our first stage is to associate a number of tours on $G_j$, in $\WWWB_{G_j}((k_i);u)$, to each tour on $G_{j+1}$, in $\WWWB_{G_{j+1}}((k_i);u)$.

To simplify the notation, let $\Hzero=G_j$ and $H=G_{j+1}$ for some $j<t$.
Let $v$ be the vertex of $\Hzero$ that is split in $H$,
and let $v'$ be the leaf vertex in $H$ added when splitting $v$.
Let $e_1$ be the (only) edge incident to $v'$ in $H$; we also use $e_1$ to refer to the corresponding edge (incident to $v$) in $\Hzero$ (see Figure~\ref{figSplittingLB}).

\begin{figure}[ht]
  $$
    \raisebox{.4in}{$H=G_{j+1}\!:\;\:$}
    \begin{tikzpicture}[scale=1.36,rotate=-36]
      \draw [thick] (1.1,0)--(.2,0) node[midway,below]{$e_1$};
      \draw [thick] (1,.2)--(.55,.2) node[very near end,right]{$\;k_1$};
      \draw [thick] (.55,.2) to [out=180,in=90](.5,.15) to [out=-90,in=180](.55,.1);
      \draw [<-,thick] (1,.1)--(.55,.1);
      \draw [thick] (0,0) -- (.31,.95);
      \draw [thick] (-.63,.36)--(0,0)--(-.57,-.57);
      \draw [thick] (0,0)--(.14,-.79);
      \draw [dashed,thick,] (1,0) to [out=0,in=-45] (1.46,1.06) to [out=135,in=72] (.31,.95);
      \draw [fill] (0,0) circle [radius=0.044];
      \draw [fill] (.2,0) circle [radius=0.044];
      \draw [fill] (1.1,0) circle [radius=0.044];
      \draw [fill] (.31,.95) circle [radius=0.044];
      \node at (-.2,0) {$v\:$};
      \node at (.26,.21) {$v'$};
    \end{tikzpicture}
    \qquad
    \raisebox{.45in}
    {\begin{tikzpicture}\draw [->,ultra thick](0,0)--(.6,0);\end{tikzpicture}}
    \qquad
    \raisebox{.4in}{$\Hzero=G_j\!:\;\:$}
    \begin{tikzpicture}[scale=1.36,rotate=-36]
      \draw [thick] (1.1,0)--(0,0) node[midway,below]{$e_1$};
      \draw [->,thick] (.95,.1)--(.45,.1);
      \draw [<-,thick] (.95,.2)--(.45,.2) node[very near end,right]{$\;k_1$};
      \draw [thick] (0,0) -- (.31,.95);
      \draw [thick] (-.63,.36)--(0,0)--(-.57,-.57);
      \draw [thick] (0,0)--(.14,-.79);
      \draw [dashed,thick,] (1,0) to [out=0,in=-45] (1.46,1.06)
        to [out=135,in=72] (.31,.95);
      \draw [fill] (0,0) circle [radius=0.044];
      \draw [fill] (1.1,0) circle [radius=0.044];
      \draw [fill] (.31,.95) circle [radius=0.044];
      \node at (-.2,0) {$v\:$};
    \end{tikzpicture}
    \vspace{-3pt}
  $$
  \caption{Tours on $\Hzero$ corresponding to a tour on $H$}\label{figSplittingLB}
\end{figure}
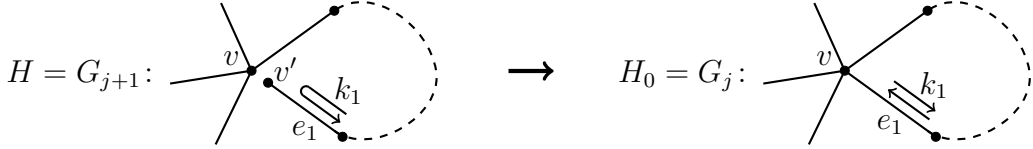
Any tour in $\WWWB_H((k_i);u)$ visits vertex $v$ exactly $\vis{H}{v}$ times and visits vertex $v'$ (along $e_1$)
$k_1$
times. The corresponding tour on $\Hzero$ visits $v$ exactly $\vis{\Hzero}{v}=\vis{H}{v}+k_1$ times. Of these visits there are $k_1$ which arrive along $e_1$ and then depart along $e_1$.

Since $\vis{\Hzero}{v}$ is positive,
separating the visits are $\vis{\Hzero}{v}-1$ excursions from $v$. Depending on whether the final visit to $v$ departs along $e_1$ or not, either $k_1-1$ or $k_1$ of these excursions begin with a traversal of $e_1$; these are interleaved with the other $\vis{H}{v}$ or $\vis{H}{v}-1$ excursions which begin with a traversal of some other edge.

Changing the interleaving of these two sets of excursions (without altering their internal ordering) produces at least
$$
\min\left[\binom{\vis{\Hzero}{v}-1}{k_1-1},\binom{\vis{\Hzero}{v}-1}{k_1}\right]
\;\geqslant\;
\binom{\vis{\Hzero}{v}-2}{k_1-1}
$$
distinct tours in $\WWWB_{\Hzero}((k_i);u)$.

Note that there is only one interleaving of the sets of excursions that corresponds to a valid tour in $\WWWB_H((k_i);u)$: the one in which the excursions beginning with a traversal of $e_1$ away from $v$ are arranged so they occur immediately following a traversal of $e_1$ towards~$v$.

Hence we can deduce that
\begin{equation}\label{eqUnsplitOnce}
\WB_{\Hzero}((k_i);u)
\;\geqslant \;
\binom{\vis{\Hzero}{v}-2}{k_1-1}\WB_H((k_i);u).
\end{equation}

\subsubsection*{II. Fully splitting one vertex}
For a given vertex $v$, let
$H_0,H_1,\ldots,H_r$
be the subsequence of graphs that corresponds to the splitting of $v$.
In our second stage, we
relate the number of tours on $H_0$ to the number of tours on $H_r$.

Note that
$\vis{H_0}{v}=\vis{G}{v}$
and
$\vis{H_r}{v}=\vis{T}{v}$
since the splitting of other vertices cannot affect the number of visits to $v$.

Let $e_1,\ldots,e_r$ be the new pendant edges in $H_r$, and hence also in $T$, added when $v$ is split, and let $e_1,\ldots,e_r$ also denote the corresponding edges in $G$.
Then
$\vis{H_{i-1}}{v}=\vis{H_i}{v}+k_i$
for $1\leqslant i\leqslant r$,
and thus
$\vis{H_{i-1}}{v}
=\vis{T}{v}+k_i+k_{i+1}+\ldots+k_r$,
and in particular
$\vis{G}{v}
=\vis{T}{v}+k_1
+\ldots+k_r$.

Hence, by iterating inequality~\eqref{eqUnsplitOnce}, 
\begin{align}\label{eqUnsplitVertex}
\WB_{H_0}((k_i);u)
\;\geqslant\;\; &
\prod_{i=1}^r\binom{\vis{H_{i-1}}{v}-2}{k_i-1}
\+\WB_{H_r}((k_i);u) \nonumber\\[10pt]
=\;\; &
\prod_{i=1}^r\binom{\vis{T}{v}+\big(\sum_{j=i}^r k_j\big)-2}{k_i-1}
\+\WB_{H_r}((k_i);u) \nonumber\\[10pt]
\geqslant\;\; &
\prod_{i=1}^r\binom{\vis{T}{v}+\big(\sum_{j=i}^r (k_j-1)\big)-1}{k_i-1}
\+\WB_{H_r}((k_i);u) \nonumber\\[10pt]
=\;\; &
\binom{\vis{G}{v}-(r+1)}{\vis{T}{v}-1,\,k_1-1,\,k_2-1,\,\ldots,\,k_r-1}
\+\WB_{H_r}((k_i);u).
\end{align}

\subsubsection*{III. Fully splitting all vertices}
Finally, our third stage is to relate the number of tours on $G$ to the number
of tours on $T$ and then apply the tree bounds to establish the required lower bound.

For each $v\in V(G)$, let $r(v)$ be the number of times $v$ is split.
Note that $r(v)$ is less than the degree of $v$ in $G$ since $\dg_G(v)=\dg_T(v)+r(v)$.

Thus,
with a suitable indexing of the edges around each vertex,
if we
iterate inequality~\eqref{eqUnsplitVertex}
and combine with the lower bound on $\WB_{T}((k_i);u)$ from Corollary~\ref{corTreeBounds}, we get
\begin{align*}
\WB_G((k_i);u)
\;\;\geqslant\;\; &
\prod_{v\in V(G)}
\binom{\vis{G}{v}-(r(v)+1)}{\vis{T}{v}-1,\,{k^v_1}-1,\,\ldots,\,{k^v_{r(v)}}-1}
\+\WB_{T}((k_i);u) \\[10pt]
\geqslant\;\; &
\prod_{v\in V(G)}
\binom{\vis{G}{v}-(r(v)+1)}{\vis{T}{v}-1,\,{k^v_1}-1,\,\ldots,\,{k^v_{r(v)}}-1}
\binom{\vis{T}{v}-\dg_T(v)}{{k^v_{r(v)+1}}\!-\!1,\,\ldots,\,{k^v_{\dg_G(v)}}\!-\!1} \\[10pt]
\;\;\geqslant\;\; &
\prod_{v\in V(G)}
\binom{\vis{G}{v}-\dg_G(v)}{{k^v_1}-1,\,{k^v_2}-1,\,\ldots,\,{k^v_{\dg_G(v)}}-1}
\end{align*}
concluding the proof of Lemma~\ref{lemmaLowerBound}.
\end{proof}

\subsection{The upper bound}\label{sectUpperBound}
Our upper bound applies to arbitrary families of tours $\WWW_G((h_i);u)$, without any restriction on the values of the $h_i$.
Subsequently, we will apply this result to families of tours of even length.
\begin{lemma}
\label{lemmaUpperBound}
If $G$ is a connected graph with $m$ edges and $u$ is any vertex of $G$,
then
$$
W_G(h_1,h_2,\ldots,h_m;u)
\; \leqslant \;
(h+2m)^m
\prod\limits_{v\in V(G)}\binom
{{k^v_1}+{k^v_2}+\ldots+{k^v_{\dg(v)}}}
{{k^v_1},\,{k^v_2},\,\ldots,\,{k^v_{\dg(v)}}}
$$
for some
$k_i\in[\half h_i, \half h_i + m]$ $(1\leqslant i\leqslant m)$,
where $h=h_1+\ldots+h_m$ is the length of the tours in the family
and
$k^v_1,k^v_2,\ldots,k^v_{\dg(v)}$ are the $k_i$
corresponding to
edges incident to $v$.
\end{lemma}
\begin{proof}
Let $T$ be some treeification of $G$ with treeification sequence $G=G_0,\ldots,G_t=T$ in which
vertex $u$ is never split.
(This is possible by positioning $u$ last in the ordering on
the vertices.)

We relate the number of (arbitrary) tours in $\WWW_G((h_i);u)$ to the number of (balanced) tours in $\WWWB_T((k_i);u)$, for some $k_i$ not much greater than $\half h_i$.
This is achieved by exhibiting a surjection from $\WWWB_T((k_i);u)$ onto $\WWW_G((h_i);u)$
that is consistent with the
homomorphism from $T$ onto $G$ induced by the treeification process.

The proof is broken down into the same three stages as for the proof of the lower bound.
Initially, we restrict ourselves to the case in which all the $\vis{G}{v}$ are positive.
The case of unvisited vertices is addressed in an additional stage at the end. 

\subsubsection*{I. Splitting once}
Our first stage is to associate to each tour on $G_j$ a number of tours on $G_{j+1}$.
However, unlike in the proof of the lower bound, the relationship is not between classes with the same parameterisation.
Rather, we relate tours in $\WWW_{G_j}((h_i);u)$
to slightly longer tours in
$\WWW_{G_{j+1}}((h'_i);u)$, for some $h'_i$ such that,
for each $i$, $h_i \leqslant h'_i \leqslant h_i + 2$.

As we did for the lower bound, let $\Hzero=G_j$ and $H=G_{j+1}$ for some $j<t$.
Let $v$ be the vertex of $\Hzero$ that is split in $H$,
and let $v'$ be the leaf vertex in $H$ added when splitting $v$.

Again, let $e_1$ be the (only) edge incident to $v'$ in $H$; we also use $e_1$ to refer to the corresponding edge (incident to $v$) in $\Hzero$.

Let
$C$ be some cycle in $\Hzero$
containing $e_1$, and let
$e_2$ be the other edge on $C$ that is incident to $v$ (in both $\Hzero$ and $H$).

\begin{figure}[ht]
  $$
    \raisebox{.4in}{$\Hzero=G_j\!:\;\:$}
    \begin{tikzpicture}[scale=1.36,rotate=-36]
      \draw [thick] (1.1,0)--(0,0) node[midway,below]{$e_1$};
      \draw [<->,thick] (.95,.1)--(.45,.1) node[very near end,right]{$\;h_1$};
      \draw [thick] (0,0) -- (.31,.95) node[midway,above]{$e_2\;$};
      \draw [thick] (-.63,.36)--(0,0)--(-.57,-.57);
      \draw [thick] (0,0)--(.14,-.79);
      \draw [dashed,thick,] (1,0) to [out=0,in=-45] (1.46,1.06) node[left]{$C$} to [out=135,in=72] (.31,.95);
      \draw [fill] (0,0) circle [radius=0.044];
      \draw [fill] (1.1,0) circle [radius=0.044];
      \draw [fill] (.31,.95) circle [radius=0.044];
      \node at (-.2,0) {$v\:$};
    \end{tikzpicture}
    \qquad
    \raisebox{.45in}
    {\begin{tikzpicture}\draw [->,ultra thick](0,0)--(.6,0);\end{tikzpicture}}
    \qquad
    \raisebox{.4in}{$H=G_{j+1}\!:\;\:$}
    \begin{tikzpicture}[scale=1.36,rotate=-36]
      \draw [thick] (1.1,0)--(.2,0) node[midway,below]{$e_1$};
      \draw [thick] (1,.2)--(.55,.2) node[very near end,right]{$\;k_1$};
      \draw [thick] (.55,.2) to [out=180,in=90](.5,.15) to [out=-90,in=180](.55,.1);
      \draw [<-,thick] (1,.1)--(.55,.1);
      \draw [thick] (0,0) -- (.31,.95) node[midway,above]{$e_2\;$};
      \draw [thick] (-.63,.36)--(0,0)--(-.57,-.57);
      \draw [thick] (0,0)--(.14,-.79);
      \draw [dashed,thick,] (1,0) to [out=0,in=-45] (1.46,1.06) to [out=135,in=72] (.31,.95);
      \draw [fill] (0,0) circle [radius=0.044];
      \draw [fill] (.2,0) circle [radius=0.044];
      \draw [fill] (1.1,0) circle [radius=0.044];
      \draw [fill] (.31,.95) circle [radius=0.044];
      \node at (-.2,0) {$v\:$};
      \node at (.26,.21) {$v'$};
    \end{tikzpicture}
    \vspace{-3pt}
  $$
  \caption{Tours on $H$ corresponding to a tour on $\Hzero$; $k_1=\floor{\half h_1}+1$}\label{figSplittingUB}
  \end{figure}

Given a tour on $\Hzero$, we want to modify it so that the result is a valid tour on $H$. For a tour on $\Hzero$ to be valid on $H$, each traversal of $e_1$ towards $v$ must be immediately followed by a traversal of $e_1$ from $v$. See Figure~\ref{figSplittingUB}.

\begin{samepage}
To achieve this, we make three kinds of changes to excursions from $v$: \vspace{-12pt}
\begin{enumerate}
  \itemsep0pt
  \item Reverse the direction of some of the excursions.
  \item Add one or two additional excursions (around $C$).
  \item Modify the interleaving of excursions.
\end{enumerate}
\vspace{-6pt}
\end{samepage}

\begin{samepage}
To manage the details, given a tour on $\Hzero$, we consider the
$\vis{\Hzero}{v}-1$
excursions from $v$ to be partitioned into subsets as follows:
\begin{center}
\begin{tabular}{rl}
  \textbf{\xx{}}: & $a_0$ excursions that don't traverse $e_1$ at all \\
  \textbf{\ox{}}: & $a_1$ excursions that begin but don't end with a traversal of $e_1$ \\
  \textbf{\xo{}}: & $a_2$ excursions that end but don't begin with a traversal of $e_1$ \\
  \textbf{\oo{}}: & $a_3$ excursions that both begin and end with traversals of $e_1$
\end{tabular}
\end{center}
\end{samepage}
We also refer to \ox{} and \oo{} excursions as \textbf{1-initial}, and \xx{} and \xo{} excursions as \textbf{\xinit{}}.

We refer to the edge traversed in arriving for the first visit to $v$ as the \textbf{arrival edge}
and to the edge traversed in departing from the last visit to $v$ as the \textbf{departure edge}.
We call their traversals the \textbf{arrival} and the \textbf{departure} respectively.
To account for these,
we define $a^+_1$ to be $a_1+1$ if the
departure edge is $e_1$
and to be $a_1$ otherwise,
and define $a^+_2$ to be $a_2+1$ if the
arrival edge is $e_1$
and to be $a_2$ otherwise.

So, to transform a tour on $\Hzero$ into one on $H$, we perform the following three steps: \vspace{-6pt}
\begin{enumerate}
  \item If $a^+_2>a_1+1$, reverse the direction of the last $\floor{\half(a^+_2-a_1)}$ of the \xo{} excursions (making them \ox{}). \\
      On the other hand, if $a^+_2<a_1$, reverse the direction of the last $\ceil{\half(a_1-a^+_2)}$ of the \ox{} excursions (making them \xo{}). \\
      Update the values of
      $a_1$ and $a_2$
      to reflect
      these reversals; we now have $a^+_2=a_1$ or $a^+_2=a_1+1$.
  \item If $a^+_1+a^+_2$ is even ($h_1$ is even) or $a^+_1=a_1$
      (the departure edge \emph{isn't} $e_1$),
      add a new \ox{} excursion consisting of a tour around the cycle $C$ (returning to $v$ along $e_2$); this should be added following all the existing excursions. \\
      Also, if $a^+_1+a^+_2$ is even
      ($h_1$ is even)
      or $a^+_1=a_1+1$
      (the departure edge \emph{is} $e_1$),
      add a new \xo{} excursion consisting of a tour around the cycle $C$ (departing from $v$ along $e_2$); this should be added following all the existing excursions. \\
      Update the values of
      $a_1$ and $a_2$
      to reflect the presence of the new excursion(s); we now have $a^+_2=a^+_1$.
  \item Change the interleaving of
      the 1-initial excursions
      with
      the \xinit{} excursions
      so that
      each visit to $v$ along $e_1$ returns immediately along $e_1$.
      This is always possible (see below) and there is only one way of doing it.
      We now have a valid tour on $H$.
\end{enumerate}
\begin{figure}[ht]
\newcommand{\gp}{@{$\:\,$}}
\newcommand{\mb}[1]{$\:\!\!$\textbf{#1}$\:\!\!$}
\newcommand{\mi}[1]{$\:\!\!$\emph{#1}$\:\!\!$}
\newcommand{\mbi}[1]{$\!\!$\textbf{\emph{#1}}$\!\!$}
\begin{center}
\begin{tabular}
{|l|c\gp|\gp c\gp|\gp c\gp|\gp c\gp|\gp c\gp|\gp c\gp|\gp c\gp|\gp c\gp|\gp c\gp|\gp c\gp|\gp c\gp|\gp c\gp|\gp c\gp|\gp c|}
\hline
      &-1&3-2&3-1&\mb{1-1}&3-1&2-1&2-1&3-2&\mb{1-3}&3-1&2-2&\mb{1-1}&&2- \\ \hline
Step~1&-1&3-2&3-1&\mb{1-1}&3-1&2-1&\mbi{1-2}&3-2&\mb{1-3}&\mbi{1-3}&2-2&\mb{1-1}&&2- \\ \hline
Step~2&-1&3-2&3-1&\mb{1-1}&3-1&2-1&\mb{1-2}&3-2&\mb{1-3}&\mb{1-3}&2-2&\mb{1-1}&\mb{1-2}&2- \\ \hline
Step~3&-1&\mb{1-1}&\mb{1-2}&3-2&3-1&\mb{1-3}&3-1&\mb{1-3}&2-1&\mb{1-1}&\mb{1-2}&3-2&2-2&2- \\ \hline
\end{tabular}
\end{center}
\begin{center}
\begin{tabular}
{|l|c\gp|\gp c\gp|\gp c\gp|\gp c\gp|\gp c\gp|\gp c\gp|\gp c\gp|\gp c\gp|\gp c\gp|\gp c\gp|\gp c\gp|\gp c\gp|\gp c\gp|\gp c\gp|\gp c\gp|\gp c|}
\hline
      &-3&\mb{1-1}&\mb{1-3}&\mb{1-3}&3-1&\mb{1-2}&3-3&\mb{1-2}&2-1&\mb{1-1}&3-1&2-3&\mb{1-1}&&&1- \\ \hline
Step~1&-3&\mb{1-1}&\mb{1-3}&\mb{1-3}&3-1&\mb{1-2}&3-3&\mi{2-1}&2-1&\mb{1-1}&3-1&2-3&\mb{1-1}&&&1- \\ \hline
Step~2&-3&\mb{1-1}&\mb{1-3}&\mb{1-3}&3-1&\mb{1-2}&3-3&2-1&2-1&\mb{1-1}&3-1&2-3&\mb{1-1}&\mb{1-2}&2-1&1- \\ \hline
Step~3&-3&3-1&\mb{1-1}&\mb{1-3}&3-3&2-1&\mb{1-3}&2-1&\mb{1-2}&3-1&\mb{1-1}&\mb{1-1}&\mb{1-2}&2-3&2-1&1- \\ \hline
\end{tabular}
\vspace{-3pt}
\end{center}
\caption{Two examples of transforming tours by modifying excursions}\label{figModifyExcursions}
\end{figure}
Figure~\ref{figModifyExcursions} shows two examples of this process. The two-digit entries in the table represent the initial and final edges traversed by excursions
from $v$;
the single-digit entries give the
arrival and departure edges;
$e_3$ is an additional edge incident to $v$.
1-initial excursions (whose interleaving with the
\xinit{}
excursions is modified by Step~3) are shown in bold.
In Step~1, excursions which are reversed are shown in italics.

\vspace{-9pt}
\subsubsection*{Validation of Step~3}
\vspace{-9pt}
If we consider the
1-initial excursions and the \xinit{} excursions as two separate lists,
with the
\xinit{} excursions
(together with the arrival and departure)
as ``fixed'', then we can insert
1-initial excursions into the list of \xinit{} excursions as follows: \vspace{-12pt}
\begin{quote}
Following
each
\xo{} excursion
(and the arrival if it is along $e_1$),
place the next unused \ox{} excursion
together with any unused 1-1 excursions that precede it.
\end{quote}
\vspace{-12pt}
This procedure is successful, and ensures that
each visit to $v$ along $e_1$ returns immediately along $e_1$ as along as
the number of traversals of $e_1$ towards $v$ equals the number of traversals of $e_1$ away from $v$, unless either
\vspace{-12pt}
\begin{itemize}
  \itemsep0pt
  \item the departure edge is \emph{not} $e_1$ and the last 1-initial excursion is a 1-1 excursion (the minimal example being~\mbox{-2~1-1~2-}, using the notation of Figure~\ref{figModifyExcursions}), or
  \item the departure edge \emph{is} $e_1$ and the last \xinit{} excursion is a \xx{} excursion
   (the minimal example being
   \mbox{-1~2-2~1-}).
\end{itemize}
\vspace{-12pt}
The rules controlling the addition of new final \xo{} and \ox{} excursions in Step~2 guarantee both that the number of traversals of $e_1$
towards $v$ is the same as the number of traversals of $e_1$ away from $v$, and also that neither of the two exceptional cases occur. Thus Step~3 is always valid.

\vspace{-9pt}
\subsubsection*{Counting}
\vspace{-9pt}
Step~2 can add at most two additional excursions from $v$ (around $C$), so
given a tour in $\WWW_{\Hzero}((h_i);u)$, this process produces a tour in $\WWW_H(2k_1,h'_2,\ldots,h'_m;u)$
where $k_1=\floor{\half h_1}+1$,
and for each $i$, $h_i \leqslant h'_i \leqslant h_i + 2$.

After completing Step~1, there are $a_1+a_2+1$ ways in which it could be undone (reverse no more than $a_1$ \ox{} excursions, reverse no more than $a_2$ \xo{} excursions, or do nothing). Since $h_1=a_1+a_2+2a_3$, this does not exceed $h_1+1$.

Also, after Step~3, there are either $k_1$ or $k_1-1$ excursions that begin with a traversal of $e_1$ that could, prior to the step, have been arbitrarily interleaved with those that don't.

Thus we see that there are
no more than
$$
(h_1+1)\max\!\left[\binom{\vis{H}{v}+k_1-1}{k_1},\binom{\vis{H}{v}+k_1-1}{k_1-1}\right]
\;\leqslant\;
2\+k_1\binom{\vis{H}{v}+k_1}{k_1}
$$
distinct tours in
$\WWW_{\Hzero}((h_i);u)$ that generate any specific tour in $\WWW_H(2k_1,h'_2,\ldots,h'_m;u)$.

Hence, 
\begin{equation}\label{eqSplitOnce}
W_{\Hzero}((h_i);u)
\;\leqslant\;
2\+k_1\binom{\vis{H}{v}+k_1}{k_1} W_H(2k_1,h'_2,\ldots,h'_m;u).
\end{equation}

Note also that either $\vis{H_0}{v}=\vis{H}{v}+k_1-2$ or $\vis{H_0}{v}=\vis{H}{v}+k_1-1$ (depending on whether $h_1$ is even or odd), and so
\begin{equation}\label{eqSplitOnceVisits}
\vis{H_0}{v}<\vis{H}{v}+k_1.
\end{equation}
Furthermore, $\vis{H}{v}$ is positive, since
the additional
excursion(s) ensure that $h'_2$ is positive.

\subsubsection*{II. Fully splitting one vertex}
For a given vertex $v$, let
$H_0,H_1,\ldots,H_r$
be the subsequence of graphs that corresponds to the splitting of $v$.
In the second stage of our proof, we
relate the number of tours on $H_0$ to the number of tours on $H_r$.

Note again that
$\vis{H_0}{v}=\vis{G}{v}$
and
$\vis{H_r}{v}=\vis{T}{v}$
since the splitting of other vertices cannot affect the number of visits to $v$.

We assume that $\vis{G}{v}$ is positive, and hence that $\vis{H_0}{v},\ldots,\vis{H_r}{v}=\vis{T}{v}$ are all positive too.

Let $e_1,\ldots,e_r$ be the new pendant edges in $T$ added when $v$ is split, and let $e_1,\ldots,e_r$ also denote the corresponding edges in $G$.
Then, by~\eqref{eqSplitOnceVisits},
for some $k_1,\ldots,k_r$ such that $\half h_i\leqslant k_i\leqslant \half h_i+i$,
we have
$
\vis{H_{i-1}}{v}\;<\;\vis{H_i}{v}+k_i$,
and thus
$$\vis{H_{i-1}}{v}
\;<\;\vis{T}{v}+k_i+
\ldots+k_r.$$
Hence, by iterating inequality~\eqref{eqSplitOnce},
if
$h'_i=2k_i$ for $1\leqslant i\leqslant r$,
then for some $h'_{r+1},\ldots,h'_m$ such that $h_i\leqslant h'_i\leqslant h_i+2\+r$,
\begin{align}\label{eqSplitVertex}
W_{H_0}((h_i);u)
\;\leqslant\;\; &
2^r
\+\Bigg(\prod_{i=1}^r k_i\+ \binom{\vis{H_i}{v}+k_i}{k_i}\!\Bigg)
\+W_{H_r}((h'_i);u)
\nonumber \\[10pt]
< \;\; &
2^r
\+\Bigg(\prod_{i=1}^r k_i\+ \binom{\vis{T}{v}+\sum_{j=i}^r k_j}{k_i}\!\Bigg)
\+W_{H_r}((h'_i);u)
\nonumber \\[10pt]
= \;\; &
2^r
\+\Big(\prod_{i=1}^rk_i\Big)
\+\binom{\vis{T}{v}+
\sum_{i=1}^r k_i
}{\vis{T}{v},\,k_1,\,\ldots,\,k_r}
\+W_{H_r}((h'_i);u).
\end{align}

\subsubsection*{III. Fully splitting all vertices}
In the third stage of the proof, we relate the number of tours on $G$ to the number
of tours on $T$ and then apply the tree bounds to establish the required upper bound
for the case in which all the $\vis{G}{v}$ are positive.

For each $v\in V(G)$, let $r(v)$ be the number of times $v$ is split.
Also, let $h=h_1+\ldots+h_m$ be the length of the tours in $\WWW_G((h_i);u)$.

Thus,
with a suitable indexing of the edges around each vertex,
if we
iterate inequality~\eqref{eqSplitVertex} and combine with the upper bound on $\WB_{T}((k_i);u)$ from Corollary~\ref{corTreeBounds},
we get,
for some $k_1,\ldots,k_m$ such that $\half h_i\leqslant k_i\leqslant \half h_i+m$,
\begin{align*}
W_G((h_i);u)
\;\leqslant\;\; &
2^m
\+\Big(\prod_{i=1}^m k_i\Big)
\prod_{v\in V(G)}
\binom{\vis{T}{v}+
\sum_{i=1}^{r(v)} k^v_i
}{\vis{T}{v},\,k^v_1,\,\ldots,\,k^v_{r(v)}}
\+\WB_T((k_i);u) \\[10pt]
\leqslant\;\; &
(h+2m)^m
\prod_{v\in V(G)}
\binom{\vis{T}{v}+
\sum_{i=1}^{r(v)} k^v_i
}{\vis{T}{v},\,k^v_1,\,\ldots,\,k^v_{r(v)}}
\+\binom{\vis{T}{v}}{{k^v_{r(v)+1}},\,\ldots,\,{k^v_{\dg_G(v)}}} \\[10pt]
=\;\; &
(h+2m)^m
\prod_{v\in V(G)}
\binom
{k^v_1+\ldots+k^v_{\dg(v)}}
{k^v_1,\,\ldots,\,k^v_{\dg(v)}},
\end{align*}
using the fact
that for each $i$, we have $k_i\leqslant \half h+m$.

\subsubsection*{IV. Unvisited vertices}
Thus we have the desired result for the case in which all the $\vis{G}{v}$ are positive.
To complete the proof, we consider families of tours in which some of the vertices are not visited.

If not all the $\vis{G}{v}$ are positive, then
let $G^+$ be the subgraph of $G$ induced by the vertices actually visited by tours in $\WWW_G((h_i);u)$.
Then
$W_G((h_i);u)=W_{G^+}((h_i);u)$. But we know that
\begin{align*}
W_{G^+}((h_i);u)
\;\leqslant\;\; &
(h+2m)^m
\prod_{v\in V(G^+)}
\binom
{k^v_1+\ldots+k^v_{\dg(v)}}
{k^v_1,\,\ldots,\,k^v_{\dg(v)}} \\[10pt]
\;\leqslant\;\; &
(h+2m)^m
\prod_{v\in V(G)}
\binom
{k^v_1+\ldots+k^v_{\dg(v)}}
{k^v_1,\,\ldots,\,k^v_{\dg(v)}}
\end{align*}
because the inclusion of
the unvisited vertices in
$V(G)\setminus V(G^+)$ cannot decrease the value of the product.
So the bound holds for any family $\WWW_G((h_i);u)$.

This concludes the proof of Lemma~\ref{lemmaUpperBound}.
\end{proof}

\subsection{Tours of even length}
In this subsection, we consider the family of all tours of \emph{even length} on a graph and prove that it grows at the same rate as the more restricted family of all balanced tours.

To do this, we make use of the fact
that the growth rate of a collection of objects does not change if we make
``small'' changes to what we are counting.
This follows directly from the definition of the growth rate.
We will also use this observation when we consider the relationship between permutation grid classes and families of tours on graphs in the next section.
\begin{obs}\label{obsGRPolySum}
If $\SSS$ is a collection of objects, containing $S_k$ objects of each size $k$, that has a finite growth rate, then for any positive polynomial $P$ and fixed non-negative integers $d_1,d_2$ with $d_1\leqslant d_2$,
$$
\limkinfty \Big(P(k)\sum_{j\,=\,k+d_1}^{k+d_2}\!S_j\Big)^{1/k}
\;=\;
\limkinfty S_k^{1/k}
\;=\;
\gr(\SSS)
.
$$
\end{obs}
We can employ our upper bound for $W_G((h_i);u)$ to
give us an upper bound for tours of a
specific even length.
We use
$W_G(h;u)$ to denote the number of tours of length $h$
starting and ending at vertex $u$.
\begin{lemma}\label{lemmaUpperBoundNew}
If $G$ is a connected graph with $m$ edges and $u$ is any vertex of $G$, then
the number of tours of length $2k$ on $G$ starting and ending at vertex $u$ is bounded above as follows:
$$
W_G(2k;u)
\;\leqslant\;
(m+1)^m\+
(2k+2m)^m
\displaystyle\sum_{j=k}^{k+m^2}
\sum\limits_{k_1+\ldots+k_m\,=\,j}
\+\+
\prod\limits_{v\in V(G)}\dbinom{{k^v_1}+
\ldots+{k^v_{\dg(v)}}}
{{k^v_1},\,
\ldots,\,{k^v_{\dg(v)}}}.
$$
\end{lemma}
\begin{proof}
From Lemma~\ref{lemmaUpperBound}, for any vertex $u$ of a graph $G$ with $m$ edges, we know that
$$
\begin{array}{rcl}
W_G(2k;u)
& = &
\displaystyle
\sum\limits_{h_1+\ldots+h_m\,=\,2k}
W_G((h_i);u) \\[15pt]
& \leqslant &
(2k+2m)^m
\displaystyle
\sum\limits_{h_1+\ldots+h_m\,=\,2k}
\+\+
\prod\limits_{v\in V(G)}\dbinom{{k^v_1}+{k^v_2}+\ldots+{k^v_{\dg(v)}}}
{{k^v_1},\,{k^v_2},\,\ldots,\,{k^v_{\dg(v)}}}
\end{array}
$$
where
each $k_i$ is dependent on the sequence $(h_i)$ with
$\half h_i\leqslant k_i\leqslant \half h_i+m$.

There are no more than
$(m+1)^m$
different values of the $h_i$ that give rise to any specific set of $k_i$, and we have
$k\leqslant k_1+\ldots+k_m\leqslant k+m^2$,
so
\[
W_G(2k;u)
\;\leqslant\;
(m+1)^m\+
(2k+2m)^m
\displaystyle\sum_{j=k}^{k+m^2}
\sum\limits_{k_1+\ldots+k_m\,=\,j}
\+\+
\prod\limits_{v\in V(G)}\dbinom{{k^v_1}+
\ldots+{k^v_{\dg(v)}}}
{{k^v_1},\,
\ldots,\,{k^v_{\dg(v)}}}.
\qedhere
\]
\end{proof}
Now, drawing together our upper and lower bounds enables us to deduce that
the family of balanced tours on a graph $G$ grows at the same rate as the family of all tours of even length on $G$.
We use $\WWWB_G$ for the family of all balanced tours on $G$ and
$\WWWE_G$ for the family of all tours of even length on $G$, where, in both cases,
we consider the size of a tour to be \emph{half} its length.
\begin{thm}\label{thmBalancedEqualsEven}
The growth rate of the family of balanced tours ($\WWWB_G$) on a connected graph is the same as growth rate of the family of all tours of even length ($\WWWE_G$) on the graph.
\end{thm}
\begin{proof}
From Lemma~\ref{lemmaLowerBound}, we know that
$$
\prod_{v\in V(G)}\binom
{{k^v_1}+
\ldots+{k^v_{\dg(v)}}}
{{k^v_1},\,
\ldots,\,{k^v_{\dg(v)}}}
\;\leqslant\;
\WB_G(k_1\!+\!1,
\ldots,k_m\!+\!1;u).
$$
Substitution in the inequality in the statement of Lemma~\ref{lemmaUpperBoundNew} then yields the following relationship between families of even-length and balanced tours:
$$
W_G(2k;u)
\;\leqslant\;
(m+1)^m\+
(2k+2m)^m
\displaystyle\sum_{j=k+m}^{k+m+m^2}
\WB_G(j;u)
$$
where
$\WB_G(j;u)$ is the number of balanced tours of length $2j$ on $G$ starting and ending at $u$.
Combining this with Observation~\ref{obsGRPolySum} and the fact that $\WB_G(k;u)\leqslant W_G(2k;u)$
produces the result $\gr(\WWWB_G)=\gr(\WWWE_G)$.
\end{proof}
Finally, before moving on to the relationship with permutation grid classes,
we
determine the value of the growth rate of
the family of even-length tours
$\WWWE_G$.
This requires only elementary algebraic graph theory.
We
recall here
the relevant 
concepts. 
The \textbf{adjacency matrix}, $A = A(G)$ of a graph $G$ has rows and columns indexed
by the vertices of $G$, with $A_{i,j}= 1$ or $A_{i,j}= 0$ according to whether vertices $i$ and $j$
are adjacent (joined by an edge) or not.
The \textbf{spectral radius} $\rho(G)$ of a graph $G$ is the largest eigenvalue (which is real and positive) of its adjacency matrix.

\begin{lemma}\label{lemmaGREven}
The growth rate of $\WWWE_G$ exists and is equal to the square of the spectral radius of $G$.
\end{lemma}
\begin{proof}
If $G$ has $n$ vertices, then
$$
W_G(2k)
\;=\;
\sum\limits_{u\in V(G)} \!W_G(2k;u)
\;=\;
\tr(A(G)^{2k})
\;=\;
\sum\limits_{i=1}^n \lambda_i^{2k},
$$
where the $\lambda_i$ are the (real) eigenvalues of $A(G)$, the adjacency matrix of $G$, since the diagonal entries of $A(G)^{2k}$ count the number of tours of length $2k$ starting at each vertex. Thus,
$$
\gr(\WWWE_G)
\;=\;
\limkinfty\Big(\sum\limits_{i=1}^n \lambda_i^{2k}\Big)^{1/k}
$$
Now the spectral radius is given by
$
\rho
=
\rho(G)
=
\max\limits_{1\leqslant i\leqslant n} \lambda_i,
$
so we can conclude that
$$
\rho^2
\;=\;
\limkinfty (\rho^{2k})^{1/k}
\;\leqslant\;
\limkinfty\Big(\sum\limits_{i=1}^n \lambda_i^{2k}\Big)^{1/k}
\;\leqslant\;
\limkinfty \big((n\rho)^{2k}\big)^{1/k}
\;=\;
\rho^2.
$$
Thus, $\gr(\WWWE_G)=\rho(G)^2$.
\end{proof}

\section{Grid classes}\label{sectGridClasses}
In this section, we prove our main theorem, that
the growth rate of a monotone grid class of permutations is
equal to the square of the spectral radius of its row-column graph.

The proof is as follows:
First, we present an explicit expression for the number of \emph{gridded permutations} of a given length.
Then, we
use this to show that the class of gridded permutations grows at the same rate as the family of tours of even length on its row-column graph.
Finally, we demonstrate that the growth rate of a grid class is the same as the growth rate of the corresponding class of gridded permutations.

Let us begin by formally introducing the relevant permutation grid class concepts.

\subsection{Notation and definitions}\label{sectGridClassDefs}
When defining grid classes, to match the way we view permutations graphically, we index matrices from the lower left corner, with the order of the indices reversed from the normal convention. For example, $M_{2,1}$ is the entry in the second column from the left in the bottom row of $M$.

\begin{figure}[ht]
  $$
  \begin{tikzpicture}[scale=0.20]
    \plotperm{8}{3,1,5,6,7,4,8,2}
      \draw[thick] (5.5,.5) -- (5.5,8.5);
      \draw[thick] (.5,2.5) -- (8.5,2.5);
  \end{tikzpicture}
  \vspace{-6pt}
  $$
  \caption{A gridding of permutation 31567482 in
  \protect\gctwo{2}{1,1}{-1,-1}
  }\label{figGridding}
\end{figure}
If $M$ is a $0/\!\pm\!1$ matrix with $t$ columns and $u$ rows, then an \textbf{$M$-gridding} of a permutation $\sigma$ of length $k$ is a pair of
sequences
$\half=c_0\leqslant c_1\leqslant \ldots\leqslant c_t=k+\half$ (the \textbf{column dividers})
and $\half=r_0\leqslant r_1\leqslant \ldots\leqslant r_u=k+\half$ (the \textbf{row dividers})
such that
for all $i\in\{0,\ldots,t\}$ and $j\in\{0,\ldots,u\}$,
$c_i-\half\in\{0,\ldots,k\}$ and $r_j-\half\in\{0,\ldots,k\}$
and
the subsequence of $\sigma$ with indices between $c_{i-1}$ and $c_i$ and values between $r_{j-1}$ and $r_j$ is
increasing if $M_{i,j}=1$,
decreasing if $M_{i,j}=-1$, and
empty if $M_{i,j}=0$.
For example, in Figure~\ref{figGridding}, $c_1=\frac{11}{2}$ and $r_1=\frac{5}{2}$.

The \textbf{grid class} $\Grid(M)$ is then defined to be the set of all permutations that have an $M$-gridding.
The \textbf{griddings} of a permutation in $\Grid(M)$ are its $M$-griddings.
We say that $\Grid(M)$ has \textbf{size} $m$ if its matrix $M$ has $m$ non-zero entries.

The concept of a grid class of permutations has been generalised, permitting arbitrary permutation classes in each cell (see Vatter~\cite{Vatter2011}). We only consider \emph{monotone} grid classes in this paper, which we simply call ``grid classes''.
An interactive demonstration of grid classes is available
online~\cite{Bevan2012}.

Sometimes we need to consider a permutation along with a specific gridding. In this case, we refer to a permutation together with an $M$-gridding as an \textbf{$M$-gridded permutation}.
We use $\Gridhash(M)$ to denote
the class of all $M$-gridded permutations, every permutation in $\Grid(M)$ being present once with each of its griddings.
We use $\Gridhash_k(M)$ for the set of $M$-gridded permutations of length $k$.

\subsubsection*{Row-column graphs}
If $M$ has $t$ rows and $u$ columns, the \textbf{row-column graph}, $G(M)$, of $\Grid(M)$ is the \emph{bipartite} graph
with vertices $r_1,\ldots,r_t,c_1,\ldots,c_u$ and an edge between $r_i$ and $c_j$ if and only if $M_{i,j}\neq0$ (see Figure~\ref{figRowColumnGraph} for an example).
Note that any bipartite graph is
the row-column graph of some grid class, and that the size (number of edges) of the row-column graph is the same as the size (number of non-zero cells) of the grid class.
\begin{figure}[ht]
  $$
  \gclass[0.3125]{4}{3}{
    \gcrow{2}{-1, 1,0, 0}
    \gcrow{1}{ 0, 1,0,-1}
    \gcrow{0}{ 1,-1,1, 0}
  }
  \qquad\quad
  \raisebox{-.15in}{
  \begin{tikzpicture}[scale=0.5]
      \draw [thick] (4,0)--(1,0)--(1,1)--(2,1)--(2,0);
      \draw [thick] (0,0)--(1,0);
        \fill[radius=0.17] (0,0) circle ;
        \fill[,radius=0.17] (1,0) circle ;
        \fill[radius=0.17] (2,0) circle;
        \fill[,radius=0.17] (3,0) circle;
        \fill[radius=0.17] (4,0) circle;
        \fill[,radius=0.17] (2,1) circle ;
        \fill[radius=0.17] (1,1) circle ;
        \node[below]at(0,-.1){$c_3$};
        \node[below]at(1,-.1){$r_3$};
        \node[below]at(2,-.1){$c_2$};
        \node[below]at(3,-.1){$r_2$};
        \node[below]at(4,-.1){$c_4$};
        \node[above]at(1,1.1){$c_1$};
        \node[above]at(2,1.1){$r_1$};
    \end{tikzpicture}}
    \vspace{-6pt}
    $$
    \caption{A grid class and its row-column graph}
    \label{figRowColumnGraph}
\end{figure}
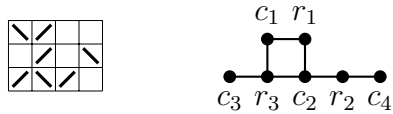

The row-column graph of a grid class captures a good deal of structural information about the class, so it is common to apportion properties of the row-column graph directly to the grid class itself, for example speaking of a
connected, acyclic or unicyclic
grid class
rather than of a grid class whose row-column graph is
connected, acyclic or unicyclic.
We follow this convention.

\subsection{Counting gridded permutations}\label{sectCountingGriddedPerms}
It is possible to give an explicit expression for the number of
gridded permutations of length $k$ in any specified grid class.
Observe the similarity to the formulae for numbers of tours.

\begin{lemma}\label{lemmaCountGriddings}
If $G=G(M)$ is the row-column graph of
$\Grid(M)$,
where $G$ has $m$ edges
$e_1,\ldots,e_m$,
then the number of gridded permutations of length $k$ in $\Gridhash(M)$ is
given by
$$
|\Gridhash_k(M)|
\;=\;
\sum_{k_1+\ldots+k_m\,=\,k}
\+\+
\prod_{v\in V(G)}\binom{{k^v_1}+{k^v_2}+\ldots+{k^v_{\dg(v)}}}
{{k^v_1}\!,\,\,{k^v_2}\!,\,\,\ldots,\,\,{k^v_{\dg(v)}}}
$$
where $k^v_1,k^v_2,\ldots,k^v_{\dg(v)}$ are the $k_i$
corresponding to
edges incident to $v$ in $G$.
\end{lemma}
\begin{proof}
A gridded permutation in $\Gridhash(M)$ consists of a number of points in each of the cells that correspond to a non-zero entry of $M$. For every permutation, the relative ordering of points (increasing or decreasing) within a particular cell is fixed by the value of the corresponding matrix entry.
However, the relative interleaving between points in distinct cells in the same row or column
can be chosen arbitrarily and independently for each row and column.

Now, each vertex in $G$ corresponds to a row or column in $M$, with an incident edge for each non-zero entry in that row or column. Thus,
the number of gridded permutations with $k_i$ points in the cell corresponding to edge $e_i$ for each $i$ is given by the following product of multinomial coefficients:
$$
\prod_{v\in V(G)}\binom{{k^v_1}+{k^v_2}+\ldots+{k^v_{\dg(v)}}}
{{k^v_1}\!,\,\,{k^v_2}\!,\,\,\ldots,\,\,{k^v_{\dg(v)}}}.
$$
The result follows by summing over values of $k_i$ that sum to $k$.
\end{proof}
As an immediate consequence, we have the fact that
the enumeration of a class of gridded permutations depends only on its row-column graph:
\begin{cor}\label{corGRHashEqForSameRCGraph}
If $G(M)=G(M')$, then
$\Gridhash_k(M) = \Gridhash_k(M')$ for all $k$.
\end{cor}

\subsection{Gridded permutations and tours}\label{sectPermsAndTours}
We now
use Lemmas~\ref{lemmaLowerBound} and~\ref{lemmaUpperBoundNew} to
relate the number of gridded permutations of length $k$ in $\Gridhash(M)$ to the number of tours of length $2k$ on $G(M)$. We restrict ourselves to permutation classes with connected row-column graphs.

\begin{lemma}\label{lemmaGRsEqual}
If $G(M)$ is connected, the growth rate of $\Gridhash(M)$ exists and is equal to the growth rate of $\WWWE_{G(M)}$.
\end{lemma}
\begin{proof}
If matrix $M$ has $m$ non-zero entries (and thus $G(M)$ has $m$ edges), then
for any vertex $u$ of $G(M)$,
combining
Lemmas~\ref{lemmaCountGriddings} and~\ref{lemmaLowerBound},
gives us
\begin{align}\label{eqLB}
|\Gridhash_k(M)| &
\;\leqslant\;
\displaystyle
\sum_{k_1+\ldots+k_m\,=\,k}
\WB_{G(M)}(k_1+1,k_2+1,\ldots,k_m+1;u) \nonumber\\[4pt]
& \;\leqslant\;
\displaystyle
\sum_{k_1+\ldots+k_m\,=\,k}
W_{G(M)}(2k_1+2,2k_2+2,\ldots,2k_m+2;u) \nonumber\\[4pt]
& \;\leqslant\;
W_{G(M)}(2k+2m).
\end{align}


On the other hand, from Lemma~\ref{lemmaUpperBoundNew}, for any vertex $u$ of a graph $G$ with $m$ edges, 
$$
W_G(2k;u)
\;\leqslant\;
(m+1)^m\+
(2k+2m)^m
\displaystyle\sum_{j=k}^{k+m^2}
\sum\limits_{k_1+\ldots+k_m\,=\,j}
\+\+
\prod\limits_{v\in V(G)}\dbinom{{k^v_1}+
\ldots+{k^v_{\dg(v)}}}
{{k^v_1},\,
\ldots,\,{k^v_{\dg(v)}}}.
$$

Let $W_G(h)$ be the number of tours of length $h$ on
$G$ (starting at any vertex).

Now $W_G(h)=\sum\limits_{u\in V(G))}\!\!W_G(h;u)$, so, using Lemma~\ref{lemmaCountGriddings}, if $G(M)$ has $n$ vertices and $m$ edges, we have
$$
W_{G(M)}(2k)
\;\leqslant\;
n\+
(m+1)^m\+
(2k+2m)^m
\displaystyle\sum_{j=k}^{k+m^2}
|\Gridhash_j(M)|.
$$
The multiplier on the right side of this inequality is a polynomial in $k$.
Hence, using inequality~\eqref{eqLB} and Observation~\ref{obsGRPolySum}, we can conclude that
\[
\gr(\Gridhash(M))\;=\;\gr(\WWWE_{G(M)})
\]
if $G(M)$ is connected.
\end{proof}

\subsection{Counting permutations}\label{sectCountingPerms}
We nearly have the result we want.
The final link is
the following lemma of Vatter which tells us that, as far as growth rates are concerned,
classes of gridded permutations are indistinguishable from their grid classes.
\begin{lemma}[Vatter~\cite{Vatter2011} Proposition~2.1]\label{lemmaGRGriddings}
The growth rate of a monotone grid class $\Grid(M)$ exists and is equal to the growth rate of the corresponding class of gridded permutations $\Gridhash(M)$.
\end{lemma}
\begin{proof}
Suppose that $M$ has dimensions $r\times s$.
Every permutation in $\Grid(M)$ has at least one gridding in $\Gridhash(M)$, but no permutation in $\Grid(M)$ of length $k$ can have more than $P(k)=\binom{k+r-1}{r-1}\binom{k+s-1}{s-1}$ griddings in $\Gridhash(M)$ because $P(k)$ is the number of possible choices for the
row and column dividers (see Subsection~\ref{sectGridClassDefs}).
Since $P(k)$ is a polynomial
in $k$, the result follows from Observation~\ref{obsGRPolySum}.
\end{proof}

Thus, by Corollary~\ref{corGRHashEqForSameRCGraph}:
\begin{cor}\label{corGREqForSameRCGraph}
Monotone grid classes with the same row-column graph have the same growth rate.
\end{cor}

\subsection{The growth rate of grid classes}
\label{sectProofOfTheorem}
We now have all we need for the proof of our main theorem.
\begin{thm}\label{thmGrowthRate}
The growth rate of a monotone grid class of permutations exists and is equal to the square of the spectral radius of its row-column graph.
\end{thm}
\begin{proof}
For connected grid classes, the result follows immediately from Lemmas~\ref{lemmaGREven},~\ref{lemmaGRsEqual} and~\ref{lemmaGRGriddings}.
A little more work is required to handle the disconnected case.

If $G(M)$ is disconnected,
then the growth rate of $\Grid(M)$ is the maximum of the growth rates of the grid classes corresponding to the connected components of $G(M)$
(see Proposition 2.10 in Vatter~\cite{Vatter2011}).

Similarly, the spectrum of a disconnected graph is the union (with multiplicities) of the spectra of the graph's components (see Theorem 2.1.1 in Cvetkovi\'c, Rowlinson and Simi\'c~\cite{CRS2010}). Thus the spectral radius of a disconnected graph is the maximum of the spectral radii of its components.

Combining these facts with
Lemmas~\ref{lemmaGREven},~\ref{lemmaGRsEqual} and~\ref{lemmaGRGriddings}
yields
$$ 
\gr(\Grid(M))\;=\;\rho(G(M))^2
$$ 
as required.
\end{proof}

\section{Implications}\label{sectConsequences}
As a consequence of Theorem~\ref{thmGrowthRate}, results concerning the spectral radius of graphs can be translated
into facts about the growth rates of permutation grid classes.
So we now present a number of corollaries
that follow from spectral graph theoretic considerations.
The two recent monographs by Cvetkovi\'c, Rowlinson and Simi\'c~\cite{CRS2010}
and Brouwer and Haemers~\cite{BH2012}
provide a valuable overview of spectral graph theory, so, where appropriate, we cite the relevant sections of these (along with the original reference for a result).

As a result of
Corollary~\ref{corGREqForSameRCGraph},
changing the sign of non-zero entries in matrix $M$ has no effect on the growth rate of $\Grid(M)$.
For this reason, when considering particular collections of grid classes below, we choose to represent them by \textbf{grid diagrams} in which non-zero matrix entries are represented by a $\!\gxone{1}{1}\!$. As with grid classes, we freely apportion properties of a row-column graph to corresponding grid diagrams.

\begin{figure}[ht]
    $$
    \gxthree{3}{1,1,1}{1,1,0}{1,0,0}
    \quad
    \gxthree{3}{1,1,1}{1,1,0}{0,1,0}
    \quad
    \gxthree{3}{1,1,1}{1,0,1}{1,0,0}
    \quad
    \gxthree{3}{1,1,1}{1,0,0}{1,0,1}
    \quad
    \gxthree{3}{1,1,1}{0,1,0}{1,1,0}
    \quad
    \gxthree{3}{1,1,0}{1,1,1}{0,1,0}
    \quad\quad\quad
    \raisebox{-.07in}{\begin{tikzpicture}[scale=0.4]
      \draw [thick] (2,1)--(0,1)--(0,0)--(2,0);
      \draw [thick] (1,0)--(1,1);
      \draw [fill] (0,0) circle [radius=0.15];
      \draw [fill] (1,0) circle [radius=0.15];
      \draw [fill] (2,0) circle [radius=0.15];
      \draw [fill] (0,1) circle [radius=0.15];
      \draw [fill] (1,1) circle [radius=0.15];
      \draw [fill] (2,1) circle [radius=0.15];
    \end{tikzpicture}}
    $$
\caption{Some unicyclic grid diagrams that have the same row-column graph}\label{figGridDiags}
\end{figure}
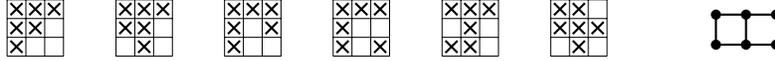
Since transposing a matrix or
permuting its rows and columns does not change the row-column graph of its grid class, there may be a number of distinct grid diagrams
corresponding to a specific row-column graph (see Figure~\ref{figGridDiags} for an example).

In many cases, we illustrate a result by showing a row-column graph and a corresponding grid diagram.
We display just one of the possible grid diagrams corresponding to the row-column graph.

Our first result is the following elementary observation, which
specifies a limitation on which numbers can be grid class growth rates.
This is a consequence of the fact that the spectral radius of a graph is
a root of the characteristic polynomial of an integer matrix.
\begin{cor}\label{corAlgebraicInteger}
The growth rate of a monotone grid class is an algebraic integer (the root of a monic polynomial).
\end{cor}

\subsection{Slowly growing grid classes}\label{sectSmallGridClasses}
Using results concerning graphs with small spectral radius, we can characterise grid classes with
growth rates no greater
than~$\frac{9}{2}$.
This is
similar to Vatter's characterisation of ``small'' permutation classes (with growth rate less than $\kappa\approx 2.20557$) in~\cite{Vatter2011}.

First, we recall that the growth rate of a disconnected grid class is the maximum of the growth rates of its components 
(see the proof of Theorem~\ref{thmGrowthRate}),
so we only need to consider connected grid classes.

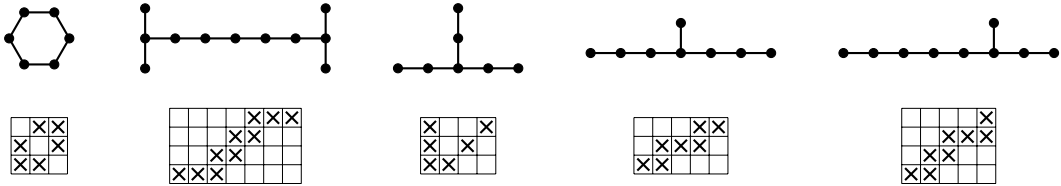
\begin{figure}[ht]
\begin{center}
\renewcommand{\arraystretch}{2.1}
\begin{tabular}{c@{$\qquad$}c@{$\qquad$}c@{$\qquad$}c@{$\qquad$}c}
  \raisebox{0.02in}{\begin{tikzpicture}[scale=0.4]
    \draw [thick] (-.5,.866)--(0,0)--(-.5,-.866)--(-1.5,-.866)--(-2,0)--(-1.5,.866)--(-.5,.866);
    \draw [fill] (0,0) circle [radius=0.15];
    \draw [fill] (-.5,.866) circle [radius=0.15];
    \draw [fill] (-.5,-.866) circle [radius=0.15];
    \draw [fill] (-1.5,.866) circle [radius=0.15];
    \draw [fill] (-1.5,-.866) circle [radius=0.15];
    \draw [fill] (-2,0) circle [radius=0.15];
  \end{tikzpicture}}
&
  \begin{tikzpicture}[scale=0.4]
    \draw [thick] (0,0)--(6,0);
    \draw [thick] (0,1)--(0,0)--(0,-1);
    \draw [thick] (6.0,1)--(6,0)--(6.0,-1);
    \foreach \x in {0,...,6}
      \draw [fill] (\x,0) circle [radius=0.15];
    \draw [fill] (0,1) circle [radius=0.15];
    \draw [fill] (0,-1) circle [radius=0.15];
    \draw [fill] (6,1) circle [radius=0.15];
    \draw [fill] (6,-1) circle [radius=0.15];
  \end{tikzpicture}
&
  \begin{tikzpicture}[scale=0.4]
    \draw [thick] (0,0)--(4,0);
    \draw [thick] (2,0)--(2,2);
    \foreach \x in {0,...,4}
      \draw [fill] (\x,0) circle [radius=0.15];
    \draw [fill] (2,1) circle [radius=0.15];
    \draw [fill] (2,2) circle [radius=0.15];
  \end{tikzpicture}
&
  \raisebox{0.08in}{\begin{tikzpicture}[scale=0.4]
    \draw [thick] (0,0)--(6,0);
    \draw [thick] (3,0)--(3,1);
    \foreach \x in {0,...,6}
      \draw [fill] (\x,0) circle [radius=0.15];
    \draw [fill] (3,1) circle [radius=0.15];
  \end{tikzpicture}}
&
  \raisebox{0.08in}{\begin{tikzpicture}[scale=0.4]
    \draw [thick] (0,0)--(7,0);
    \draw [thick] (5,0)--(5,1);
    \foreach \x in {0,...,7}
      \draw [fill] (\x,0) circle [radius=0.15];
    \draw [fill] (5,1) circle [radius=0.15];
  \end{tikzpicture}}
\\
  \gxthree{3}{0,1,1}{1,0,1}{1,1,0}
&
  \gxfour{7}{0,0,0,0,1,1,1}{0,0,0,1,1,0,0}{0,0,1,1,0,0,0}{1,1,1,0,0,0,0}
&
  \gxthree{4}{1,0,0,1}{1,0,1,0}{1,1,0,0}
&
  \gxthree{5}{0,0,0,1,1}{0,1,1,1,0}{1,1,0,0,0}
&
  \gxfour{5}{0,0,0,0,1}{0,0,1,1,1}{0,1,1,0,0}{1,1,0,0,0}
\end{tabular}
\vspace{-6pt}
\end{center}
\caption{
A cycle, an $H$ graph and the three
other Smith graphs,
with corresponding grid diagrams
}\label{figSmith}
\end{figure}
The connected graphs with spectral radius 2 are known as the \textbf{Smith} graphs. These are precisely the \emph{cycle} graphs, the $H$ graphs (paths with two pendant edges
attached to both endvertices, including the star graph $K_{1,4}$), and the three other graphs shown in Figure~\ref{figSmith}.

\begin{figure}[ht]
\begin{center}
\renewcommand{\arraystretch}{2.1}
\begin{tabular}{c@{$\qquad$}c@{$\qquad$}c@{$\qquad$}c@{$\qquad$}c}
  \raisebox{0.14in}{\begin{tikzpicture}[scale=0.4]
    \draw [thick] (0,0)--(6,0);
    \foreach \x in {0,...,6}
      \draw [fill] (\x,0) circle [radius=0.15];
  \end{tikzpicture}}
&
  \begin{tikzpicture}[scale=0.4]
    \draw [thick] (0,0)--(5,0);
    \draw [thick] (-.5,.866)--(0,0)--(-.5,-.866);
    \foreach \x in {0,...,5}
      \draw [fill] (\x,0) circle [radius=0.15];
    \draw [fill] (-.5,.866) circle [radius=0.15];
    \draw [fill] (-.5,-.866) circle [radius=0.15];
  \end{tikzpicture}
&
  \raisebox{0.06in}{\begin{tikzpicture}[scale=0.4]
    \draw [thick] (0,0)--(4,0);
    \draw [thick] (2,0)--(2,1);
    \foreach \x in {0,...,4}
      \draw [fill] (\x,0) circle [radius=0.15];
    \draw [fill] (2,1) circle [radius=0.15];
  \end{tikzpicture}}
&
  \raisebox{0.06in}{\begin{tikzpicture}[scale=0.4]
    \draw [thick] (0,0)--(5,0);
    \draw [thick] (3,0)--(3,1);
    \foreach \x in {0,...,5}
      \draw [fill] (\x,0) circle [radius=0.15];
    \draw [fill] (3,1) circle [radius=0.15];
  \end{tikzpicture}}
&
  \raisebox{0.06in}{\begin{tikzpicture}[scale=0.4]
    \draw [thick] (0,0)--(6,0);
    \draw [thick] (4,0)--(4,1);
    \foreach \x in {0,...,6}
      \draw [fill] (\x,0) circle [radius=0.15];
    \draw [fill] (4,1) circle [radius=0.15];
  \end{tikzpicture}}
\\
  \gxthree{4}{0,0,1,1}{0,1,1,0}{1,1,0,0}
&
  \gxfour{5}{0,0,0,0,1}{0,0,0,1,1}{0,0,1,1,0}{1,1,1,0,0}
&
  \gxthree{3}{0,1,1}{0,1,0}{1,1,0}
&
  \gxthree{4}{0,0,0,1}{0,1,1,1}{1,1,0,0}
&
  \gxfour{4}{0,0,1,1}{0,0,1,0}{0,1,1,0}{1,1,0,0}
\end{tabular}
\vspace{-6pt}
\end{center}
\caption{
A path, a $Y$ graph and the three
other connected proper subgraphs of Smith graphs,
with corresponding grid diagrams
}\label{figSubSmith}
\end{figure}
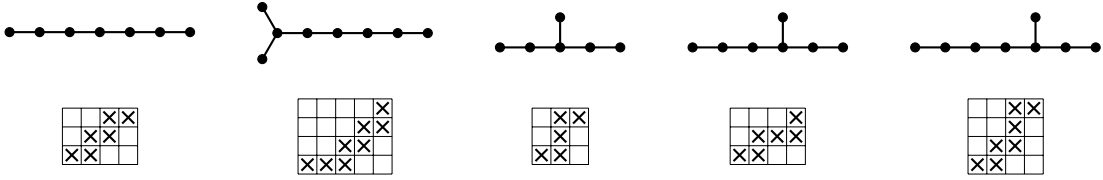
Similarly, the connected proper
subgraphs of the Smith graphs are precisely the \emph{path} graphs, the $Y$ graphs (paths with two pendant edges attached to one endvertex)
and the three other graphs in Figure~\ref{figSubSmith}.
For details, see
Smith~\cite{Smith1970} and Lemmens and Seidel~\cite{LS1973}; 
also see~\cite{CRS2010} Theorem 3.11.1 and~\cite{BH2012} Theorem 3.1.3.

With these, we can characterise all grid classes with growth rate no greater than 4:
\begin{cor}\label{corSmith}
If the growth rate of a connected monotone grid class equals 4, then its row-column graph is a Smith graph.
If the growth rate of a connected monotone grid class is less than 4, then its row-column graph is a connected proper
subgraph of a Smith graph.
\end{cor}
In particular, we have the following:
\begin{cor}\label{corCycle}
A monotone grid class
of any size
whose row-column graph is a cycle
or
an H graph
has growth rate
4.
\end{cor}

In Appendix~A of~\cite{Vatter2011}, Vatter considers \textbf{staircase} grid classes, whose row-column graphs are paths (see the leftmost grid diagram in Figure~\ref{figSubSmith}).
The 
spectral radius of a path graph has long been known (Lov\'asz and Pelik\'an~\cite{LP1973}; also see~\cite{CRS2010} Theorem 8.1.17 and~\cite{BH2012} 1.4.4), from which we can conclude:
\begin{cor}\label{corPath}
A monotone grid class
of size $m$
(having $m$ non-zero cells)
whose row-column graph is a path
has growth rate $4\cos^2\!\left(\frac{\pi}{m+2}\right)$.
This is minimal for any connected grid class of size $m$.
\end{cor}
A $Y$ graph
of size $m$
has spectral radius
$2\cos\!\left(\frac{\pi}{2m}\right)$,
and the spectral radii of
the three other graphs at the right of Figure~\ref{figSubSmith} are $2\cos\!\left(\frac{\pi}{12}\right)$,
$2\cos\!\left(\frac{\pi}{18}\right)$,
and
$2\cos\!\left(\frac{\pi}{30}\right)$,
from left to right (see~\cite{BH2012} 3.1.1).
Thus
we have the following characterisation of
growth rates less than~4:
\begin{cor}\label{corSmallGrowthRates}
If the growth rate of a monotone grid class is less than 4, it is equal to $4\cos^2\!\left(\frac{\pi}{k}\right)$ for some $k\geqslant 3$.
\end{cor}
The only grid class growth rates
no greater than 3 are $1$, $2$, $\half(3+\sqrt{5})\approx2.618$, and $3$.

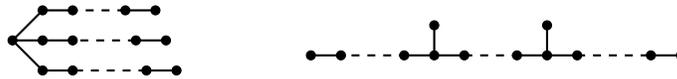
\begin{figure}[ht]
  $$
  \raisebox{-.08in}
    {\begin{tikzpicture}[scale=0.4]
    \draw [thick] (0,0)--(2,0);
    \draw [thick,dashed] (4.1,0)--(2,0);
    \draw [thick] (4.1,0)--(5.1,0);
    \draw [thick] (2,1)--(1,1)--(0,0)--(1,-1)--(2,-1);
    \draw [thick] (3.75,1)--(4.75,1);
    \draw [thick] (4.45,-1)--(5.45,-1);
    \draw [thick,dashed] (3.75,1)--(2,1);
    \draw [thick,dashed] (4.45,-1)--(2,-1);
    \draw [fill] (0,0) circle [radius=0.15];
    \foreach \x in {1,2,3.75,4.75} \draw [fill] (\x,1) circle [radius=0.15];
    \foreach \x in {1,2,4.1,5.1} \draw [fill] (\x,0) circle [radius=0.15];
    \foreach \x in {1,2,4.45,5.45} \draw [fill] (\x,-1) circle [radius=0.15];
  \end{tikzpicture}}
  \quad\quad\quad\quad
  \begin{tikzpicture}[scale=0.4]
    \draw [thick] (-0.35,0)--(0.65,0);
    \draw [thick,dashed] (1,0)--(2.75,0);
    \draw [thick] (2.75,0)--(4.75,0);
    \draw [thick] (3.75,0)--(3.75,1);
    \draw [thick,dashed] (4.75,0)--(6.5,0);
    \draw [thick] (6.5,0)--(8.5,0);
    \draw [thick] (7.5,0)--(7.5,1);
    \draw [thick,dashed] (8.5,0)--(10.95,0);
    \draw [thick] (10.95,0)--(11.95,0);
    \foreach \x in {-0.35,0.65,2.75,3.75,4.75,6.5,7.5,8.5,10.95,11.95} \draw [fill] (\x,0) circle [radius=0.15];
    \draw [fill] (3.75,1) circle [radius=0.15];
    \draw [fill] (7.5,1) circle [radius=0.15];
  \end{tikzpicture}
  $$
\caption{$E$ and $F$ graphs}\label{figEFGraphs}
\end{figure}
In order to characterise grid classes with growth rates slightly greater than 4,
let
an $E$ graph be a tree consisting of three paths having one endvertex in common, and an $F$ graph be a tree consisting of a path with a pendant edge attached to each of two distinct internal vertices (see Figure~\ref{figEFGraphs}).
Then,
results of
Brouwer and Neumaier~\cite{BN1989} and Cvetkovi\'{c}, Doob and Gutman~\cite{CDG1982}
imply the following (also see~\cite{CRS2010} Theorem 3.11.2):
\begin{cor}
If a connected monotone grid class has growth rate
between 4 and $2+\sqrt{5}$,
then its row-column graph is an $E$ or $F$ graph.
\end{cor}

Thus, since $\sqrt{2+\sqrt{5}}$ cannot be an eigenvalue of any graph (see~\cite{CRS2010} p.~93), we can deduce the following:
\begin{cor}\label{corProperCycle}
If a monotone grid class properly contains a cycle then its growth rate exceeds $2+\sqrt{5}$.
\end{cor}
More recently, Woo and Neumaier~\cite{WN2007}
have investigated the structure of graphs with spectral radius no greater than $\frac{3}{2}\sqrt{2}$
(also see~\cite{CRS2010} Theorem 3.11.3).
As a consequence,
we have the following:
\begin{cor}
If the growth rate of
a connected monotone grid class is no greater than $\frac{9}{2}$, then its row-column graph is one of the following: \vspace{-12pt}
\begin{enumerate}[~~~~~(a)]
  \itemsep0pt
  \item a tree of maximum degree 3 such that all vertices of degree 3 lie on a path,
  \item a unicyclic graph of maximum degree 3 such that all vertices of degree 3 lie on the cycle, or
  \item a tree consisting of a path with three pendant edges attached to one endvertex.
\end{enumerate}
\end{cor}

\subsection{Accumulation points of grid class growth rates}\label{sectAccumulationPoints}
Using graph theoretic results of Hoffman and Shearer, it is possible to characterise \emph{all} accumulation points of grid class growth rates.

As we have seen, the growth rates of grid classes whose row-column graphs are paths and $Y$ graphs grow to 4 from below; 4 is the least accumulation point of growth rates. The following characterises all
accumulation points
below $2+\sqrt{5}$ (see Hoffman~\cite{Hoffman1972}):
\begin{cor}
For $k=1,2,\ldots$, let $\beta_k$ be the positive root of
$$P_k(x)=x^{k+1}-(1+x+x^2+\ldots+x^{k-1})$$
and let
$\gamma_k=2+\beta_k+\beta_k^{-1}$.
Then $4=\gamma_1<\gamma_2<\ldots$ are all the accumulation points of growth rates of
monotone grid classes smaller than $2+\sqrt{5}$.
\end{cor}
The approximate values of
the first eight accumulation points
are:
4, 4.07960, 4.14790, 4.18598, 4.20703, 4.21893, 4.22582, 4.22988.

At $2+\sqrt{5}$, things change dramatically; from this value upwards grid class growth rates are dense (see Shearer~\cite{Shearer1989}):
\begin{cor}\label{corAccumulationPoints}
Every $\gamma\geqslant2+\sqrt{5}$ is an accumulation point of growth rates of
monotone grid classes.
\end{cor}

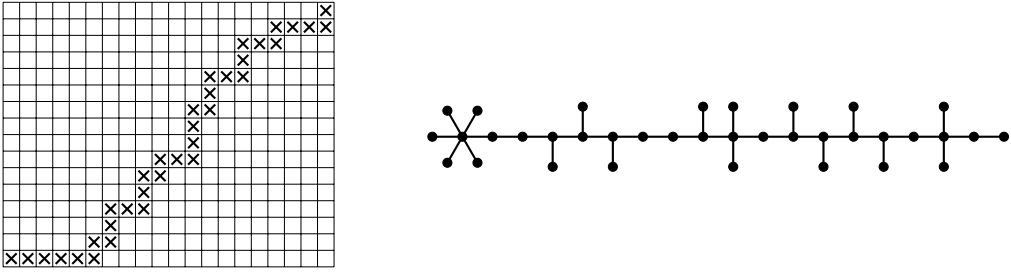
\begin{figure}[ht]
$$
\gclass[0.22]{20}{16}{
\gxrow{15}{0,0,0,0,0,0,0,0,0,0,0,0,0,0,0,0,0,0,0,1}
\gxrow{14}{0,0,0,0,0,0,0,0,0,0,0,0,0,0,0,0,1,1,1,1}
\gxrow{13}{0,0,0,0,0,0,0,0,0,0,0,0,0,0,1,1,1,0,0,0}
\gxrow{12}{0,0,0,0,0,0,0,0,0,0,0,0,0,0,1,0,0,0,0,0}
\gxrow{11}{0,0,0,0,0,0,0,0,0,0,0,0,1,1,1,0,0,0,0,0}
\gxrow{10}{0,0,0,0,0,0,0,0,0,0,0,0,1,0,0,0,0,0,0,0}
\gxrow{ 9}{0,0,0,0,0,0,0,0,0,0,0,1,1,0,0,0,0,0,0,0}
\gxrow{ 8}{0,0,0,0,0,0,0,0,0,0,0,1,0,0,0,0,0,0,0,0}
\gxrow{ 7}{0,0,0,0,0,0,0,0,0,0,0,1,0,0,0,0,0,0,0,0}
\gxrow{ 6}{0,0,0,0,0,0,0,0,0,1,1,1,0,0,0,0,0,0,0,0}
\gxrow{ 5}{0,0,0,0,0,0,0,0,1,1,0,0,0,0,0,0,0,0,0,0}
\gxrow{ 4}{0,0,0,0,0,0,0,0,1,0,0,0,0,0,0,0,0,0,0,0}
\gxrow{ 3}{0,0,0,0,0,0,1,1,1,0,0,0,0,0,0,0,0,0,0,0}
\gxrow{ 2}{0,0,0,0,0,0,1,0,0,0,0,0,0,0,0,0,0,0,0,0}
\gxrow{ 1}{0,0,0,0,0,1,1,0,0,0,0,0,0,0,0,0,0,0,0,0}
\gxrow{ 0}{1,1,1,1,1,1,0,0,0,0,0,0,0,0,0,0,0,0,0,0}
}
\quad\quad\quad
\raisebox{.5in}{\begin{tikzpicture}[scale=0.4]
  \draw [thick] (0,0)--(19,0);
  \foreach \x in {0,...,19} \draw [fill] (\x,0) circle [radius=0.15];
  \foreach \x in {4,6,10,13,15,17}
  { \draw [fill] (\x,-1) circle [radius=0.15]; \draw [thick] (\x,0)--(\x,-1); }
  \foreach \x in {10,17,5,9,12,14}
  { \draw [fill] (\x,1) circle [radius=0.15]; \draw [thick] (\x,0)--(\x,1); }
  \foreach \x in {0.5,1.5}
  {
    \draw [fill] (\x,.866) circle [radius=0.15];
    \draw [fill] (\x,-.866) circle [radius=0.15];
    \draw [thick] (1,0)--(\x,.866);
    \draw [thick] (1,0)--(\x,-.866);
  }
\end{tikzpicture}}
$$
\caption{A grid diagram
whose growth rate differs from $2\pi$ by less than $10^{-7}$,
and its caterpillar row-column graph}\label{figCaterpillar}
\end{figure}
Thus, for every $\gamma\geqslant2+\sqrt{5} \approx4.236068$,
there is a
grid class with growth rate arbitrarily close to $\gamma$.
Indeed, for $\gamma>2+\sqrt{5}$, Shearer's proof provides an iterative process for generating a sequence of grid classes,
each with a row-column graph that
is a \emph{caterpillar}
(a tree such that all vertices of degree 2 or more lie on a path), with growth rates converging to $\gamma$ from below.
An example is shown in Figure~\ref{figCaterpillar}.

\subsection{Increasing the size of a grid class}
We now consider
the effect on the growth rate of making
small changes to a grid class.

Firstly, growth rates of connected grid classes satisfy a \emph{strict} monotonicity condition (see~\cite{CRS2010}
Proposition 1.3.10):
\begin{cor}\label{corAddCell}
Adding a non-zero cell to a connected monotone grid class
while preserving connectivity
increases its growth rate.
\end{cor}
On the other hand,
particularly surprising
is the fact that grid classes with \emph{longer internal} paths or cycles have \emph{lower} growth rates.

An edge $e$ of $G$ is said to lie on an \textbf{endpath} of $G$ if $G-e$ is disconnected and one of its components is a (possibly trivial) path. An edge that does not lie on an endpath is said to be \textbf{internal}.
Note that a graph has an internal edge if and only if it contains either a cycle or non-star $H$ graph.

\begin{figure}[ht]
    $$
    \gxtwo{3}{0,1,1}{1,1,1}
    \quad
    \raisebox{-0.09in}{\begin{tikzpicture}[scale=0.4]
    \draw [thick] (0,0)--(1,0)--(1.707,0.707)--(2.414,0)--(1.707,-0.707)--(1,0);
    \draw [fill] (0,0) circle [radius=0.15];
    \draw [fill] (1,0) circle [radius=0.15];
    \draw [fill] (2.414,0) circle [radius=0.15];
    \draw [fill] (1.707,0.707) circle [radius=0.15];
    \draw [fill] (1.707,-0.707) circle [radius=0.15];
    \end{tikzpicture}}
    \quad\quad\quad\quad
    \gxthree{4}{0,1,1}{0,0,1,1}{1,1,0,1}
    \quad
    \raisebox{-0.12in}{\begin{tikzpicture}[scale=0.4]
    \draw [thick] (0,0)--(1,0)--(1.5,0.866)--(2.5,0.866)--(3,0)--(2.5,-0.866)--(1.5,-0.866)--(1,0);
    \draw [fill] (0,0) circle [radius=0.15];
    \draw [fill] (1,0) circle [radius=0.15];
    \draw [fill] (3,0) circle [radius=0.15];
    \draw [fill] (1.5,0.866) circle [radius=0.15];
    \draw [fill] (2.5,0.866) circle [radius=0.15];
    \draw [fill] (1.5,-0.866) circle [radius=0.15];
    \draw [fill] (2.5,-0.866) circle [radius=0.15];
    \end{tikzpicture}}
    \quad\quad\quad\quad
    \gxfour{5}{0,1,1}{0,0,1,1}{0,0,0,1,1}{1,1,0,0,1}
    \quad
    \raisebox{-0.19in}{\begin{tikzpicture}[scale=0.4]
    \draw [thick] (0,0)--(1,0)--(1.383,0.924)--(2.307,1.307)--(3.23,0.924)--(3.613,0)--(3.23,-0.924)--(2.307,-1.307)--(1.383,-0.924)--(1,0);
    \draw [fill] (0,0) circle [radius=0.15];
    \draw [fill] (1,0) circle [radius=0.15];
    \draw [fill] (3.613,0) circle [radius=0.15];
    \draw [fill] (1.383,0.924) circle [radius=0.15];
    \draw [fill] (1.383,-0.924) circle [radius=0.15];
    \draw [fill] (3.23,0.924) circle [radius=0.15];
    \draw [fill] (3.23,-0.924) circle [radius=0.15];
    \draw [fill] (2.307,1.307) circle [radius=0.15];
    \draw [fill] (2.307,-1.307) circle [radius=0.15];
    \end{tikzpicture}}
    \vspace{-3pt}
    $$
\caption{Three unicyclic grid diagrams, of increasing size but decreasing growth rate from left to right, and their row-column graphs}\label{figSubdivision}
\end{figure}
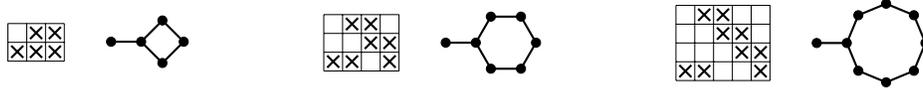
An early result of Hoffman and Smith~\cite{HS1975}
shows that the
subdivision of an internal edge \emph{reduces} the spectral radius (also see~\cite{BH2012} Proposition 3.1.4 and~\cite{CRS2010} Theorem 8.1.12). Hence, we can deduce the following unexpected consequence for grid classes:
\begin{cor}
\label{corSubdivision}
If $\Grid(M)$ is connected,
and $G(M')$ is obtained from $G(M)$ by subdividing an internal edge, then $\gr(\Grid(M'))<\gr(\Grid(M))$
unless $G(M)$ is a cycle or an $H$ graph.
\end{cor}
\vspace{-7pt}
For an example, see Figure~\ref{figSubdivision}.

\subsection{Grid classes with extremal growth rates}
Finally, we briefly consider grid classes with maximal or minimal growth rates for their size.

\begin{figure}[ht]
  $$
  \gxone{7}{1,1,1,1,1,1,1}
  \quad\quad\quad
  \raisebox{-.16in}{\begin{tikzpicture}[scale=0.40]
    \draw [thick] (.716,-.898)--(0,0)--(.716,.898);
    \draw [thick] (-.256,-1.121)--(0,0)--(-.256,1.121);
    \draw [thick] (-1.036,-.499)--(0,0)--(-1.036,.499);
    \draw [thick] (0,0)--(1,0);
    \draw [fill] (0,0) circle [radius=0.15];
    \draw [fill] (1,0) circle [radius=0.15];
    \draw [fill] (.716,-.898) circle [radius=0.15];
    \draw [fill] (-.256,-1.121) circle [radius=0.15];
    \draw [fill] (-1.036,-.499) circle [radius=0.15];
    \draw [fill] (.716,.898) circle [radius=0.15];
    \draw [fill] (-.256,1.121) circle [radius=0.15];
    \draw [fill] (-1.036,.499) circle [radius=0.15];
  \end{tikzpicture}}
  \vspace{-3pt}
  $$
\caption{A skinny grid diagram
and its row-column star graph}\label{figSkinny}
\end{figure}
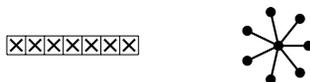
We call a grid class of a $1\times m$ matrix a \textbf{skinny} grid class. The row-column graph of a skinny grid class is a \emph{star} (see Figure~\ref{figSkinny}).
Stars have maximal spectral radius among trees (see~\cite{CRS2010} Theorem 8.1.17). This yields:
\begin{cor}\label{corSkinny}
Among all connected acyclic monotone grid classes of size $m$,
the skinny grid classes
have the largest growth rate
(equal to $m$).
\end{cor}

We have already seen (Corollary~\ref{corPath}) that the
connected grid classes with \emph{smallest} growth rates are those whose row-column graph is a path. For unicyclic grid classes, we have the following (see~\cite{CRS2010} Theorem 8.1.18):
\begin{cor}\label{corCycleMin}
Among all connected unicyclic monotone grid classes of size $m$, those
whose row-column graph is a single cycle of length $m$
have the smallest growth rate
(equal to 4).
\end{cor}
There are many additional results known concerning graphs with extremal values for their spectral radii, especially for graphs with a small number of cycles.
For an example, see the two papers by Simi\'c~\cite{Simic1987a,Simic1989}
on the largest eigenvalues of unicyclic and bicyclic graphs.
Results like these
can be translated into further facts concerning the growth rates of grid classes.

\section{Concluding remarks}\label{sectRemarks}
In light of Theorem~\ref{thmGrowthRate},
it seems likely to be worthwhile investigating
whether there are
other links between spectral graph theory and permutation grid classes, or indeed permutation classes in general.
Specifically, are there properties of grid classes which are associated with
other eigenvalues of the adjacency or Laplacian matrices of the row-column graph?
And can algebraic graph theory be used to help determine
the growth rate of an arbitrary permutation class (e.g. specified by its basis)?

Closely related to grid classes are \emph{geometric} grid classes, as investigated
by Albert, Atkinson, Bouvel, Ru\v{s}kuc and Vatter~\cite{AABRV2011}.
The geometric grid class $\Geom(M)$ is a subset of $\Grid(M)$, permutations in $\Geom(M)$ satisfying an additional ``geometric'' constraint.
Recently, the present author~\cite{Bevan2013b} has proved
a result similar to Theorem~\ref{thmGrowthRate}
for the growth rates of geometric grid classes.
Specifically, the growth rate of geometric grid class $\Geom(M)$ exists and is equal to the square of the
largest root of the \emph{matching polynomial} of the row-column graph of what is known as the ``double refinement''
of matrix $M$.
This value coincides with $\rho(G(M))^2$ for acyclic classes, $\Geom(M)$ and $\Grid(M)$ being identical when $G(M)$ is a forest.

\subsubsection*{Acknowledgements}
Grateful thanks are due to Robert Brignall for
numerous discussions related to this work, and for
much helpful advice and thorough feedback on earlier drafts of this paper,
and also to an anonymous referee whose very thorough comments led to some significant improvements in the presentation.
\begin{center}
{\footnotesize\greektext $\maltese\;\;$ MONW SOFW JEW H DOXA EIS TOUS AIWNAS $\;\;\maltese$}
\end{center}


\bibliographystyle{plain}
{\footnotesize\bibliography{mybib}}

\begin{thebibliography}{10}

\bibitem{AAB2011}
M.~H. Albert, M.~D. Atkinson, and R.~Brignall.
\newblock The enumeration of permutations avoiding 2143 and 4231.
\newblock {\em Pure Mathematics and Applications}, 22(2):87--98, 2011.

\bibitem{AAB2012}
M.~H. Albert, M.~D. Atkinson, and Robert Brignall.
\newblock The enumeration of three pattern classes using monotone grid classes.
\newblock {\em Electron. J. Combin.}, 19(3):~Research paper 20, 34 pp.
  (electronic), 2012.

\bibitem{AAB2010}
Michael Albert, Mike Atkinson, and Robert Brignall.
\newblock Notes on the basis of monotone grid classes.
\newblock Unpublished, 2010.

\bibitem{AABRV2011}
Michael~H. Albert, M.~D. Atkinson, Mathilde Bouvel, Nik Ru{\v{s}}kuc, and
  Vincent Vatter.
\newblock Geometric grid classes of permutations.
\newblock {\em Trans. Amer. Math. Soc.}, 365(11):5859--5881, 2013.

\bibitem{AAV2012}
Michael~H. Albert, M.~D. Atkinson, and Vincent Vatter.
\newblock Inflations of geometric grid classes: three case studies.
\newblock {\em Preprint.
  \href{http://arxiv.org/pdf/1209.0425}{arXiv:1209.0425}}, 2012.

\bibitem{Atkinson1998}
M.~D. Atkinson.
\newblock Permutations which are the union of an increasing and a decreasing
  subsequence.
\newblock {\em Electron. J. Combin.}, 5:~Research paper 6, 13 pp. (electronic),
  1998.

\bibitem{Atkinson1999}
M.~D. Atkinson.
\newblock Restricted permutations.
\newblock {\em Discrete Math.}, 195(1-3):27--38, 1999.

\bibitem{Bevan2012}
David Bevan.
\newblock Permutation grid classes.
\newblock
  \href{http://demonstrations.wolfram.com/PermutationGridClasses}{http:/$\!$/demonstrations.wolfram.com/PermutationGrid}
  \href{http://demonstrations.wolfram.com/PermutationGridClasses}{Classes},
  Wolfram Demonstrations Project, 2012.

\bibitem{Bevan2013b}
David Bevan.
\newblock Growth rates of geometric grid classes of permutations.
\newblock {\em Preprint.
  \href{http://arxiv.org/pdf/1306.4246}{arXiv:1306.4246}}, 2013.

\bibitem{BN1989}
A.~E. Brouwer and A.~Neumaier.
\newblock The graphs with spectral radius between {$2$} and
  {$\sqrt{2+\sqrt{5}}$}.
\newblock {\em Linear Algebra Appl.}, 114/115:273--276, 1989.

\bibitem{BH2012}
Andries~E. Brouwer and Willem~H. Haemers.
\newblock {\em Spectra of graphs}.
\newblock Universitext. Springer, New York, 2012.

\bibitem{CDG1982}
Drago{\v{s}} Cvetkovi{\'c}, Michael Doob, and Ivan Gutman.
\newblock On graphs whose spectral radius does not exceed
  {$(2+\sqrt{5})^{1/2}$}.
\newblock {\em Ars Combin.}, 14:225--239, 1982.

\bibitem{CRS2010}
Drago{\v{s}} Cvetkovi{\'c}, Peter Rowlinson, and Slobodan Simi{\'c}.
\newblock {\em An introduction to the theory of graph spectra}, volume~75 of
  {\em London Mathematical Society Student Texts}.
\newblock Cambridge University Press, Cambridge, 2010.

\bibitem{FK2008}
Ji{\v{r}}{\'{\i}} Fiala and Jan Kratochv{\'{\i}}l.
\newblock Locally constrained graph homomorphisms---structure, complexity, and
  applications.
\newblock {\em Computer Science Review}, 2(2):97--111, 2008.

\bibitem{HN2004}
Pavol Hell and Jaroslav Ne{\v{s}}et{\v{r}}il.
\newblock {\em Graphs and homomorphisms}, volume~28 of {\em Oxford Lecture
  Series in Mathematics and its Applications}.
\newblock Oxford University Press, Oxford, 2004.

\bibitem{Hoffman1972}
Alan~J. Hoffman.
\newblock On limit points of spectral radii of non-negative symmetric integral
  matrices.
\newblock In {\em Graph theory and applications}, volume 303 of {\em Lecture
  Notes in Mathematics}, pages 165--172. Springer, Berlin, 1972.

\bibitem{HS1975}
Alan~J. Hoffman and John~Howard Smith.
\newblock On the spectral radii of topologically equivalent graphs.
\newblock In {\em Recent advances in graph theory ({P}roc. {S}econd
  {C}zechoslovak {S}ympos., {P}rague, 1974)}, pages 273--281. Academia, Prague,
  1975.

\bibitem{HV2006}
Sophie Huczynska and Vincent Vatter.
\newblock Grid classes and the {F}ibonacci dichotomy for restricted
  permutations.
\newblock {\em Electron. J. Combin.}, 13(1):~Research paper 54, 14 pp.
  (electronic), 2006.

\bibitem{KSW1996}
Andr{\'e}~E. K{\'e}zdy, Hunter~S. Snevily, and Chi Wang.
\newblock Partitioning permutations into increasing and decreasing
  subsequences.
\newblock {\em J.~Combin. Theory Ser.~A}, 73(2):353--359, 1996.

\bibitem{LS1973}
P.~W.~H. Lemmens and J.~J. Seidel.
\newblock Equiangular lines.
\newblock {\em J.~Algebra}, 24:494--512, 1973.

\bibitem{LP1973}
L.~Lov{\'a}sz and J.~Pelik{\'a}n.
\newblock On the eigenvalues of trees.
\newblock {\em Period. Math. Hungar.}, 3:175--182, 1973.

\bibitem{MT2004}
Adam Marcus and G{\'a}bor Tardos.
\newblock Excluded permutation matrices and the {S}tanley-{W}ilf conjecture.
\newblock {\em J.~Combin. Theory Ser.~A}, 107(1):153--160, 2004.

\bibitem{MV2003}
Maximillian~M. Murphy and Vincent~R. Vatter.
\newblock Profile classes and partial well-order for permutations.
\newblock {\em Electron. J. Combin.}, 9(2):~Research paper 17, 30 pp.
  (electronic), 2003.

\bibitem{Shearer1989}
James~B. Shearer.
\newblock On the distribution of the maximum eigenvalue of graphs.
\newblock {\em Linear Algebra Appl.}, 114/115:17--20, 1989.

\bibitem{Simic1987a}
Slobodan~K. Simi{\'c}.
\newblock On the largest eigenvalue of unicyclic graphs.
\newblock {\em Publ. Inst. Math. (Beograd) (N.S.)}, 42(56):13--19, 1987.

\bibitem{Simic1989}
Slobodan~K. Simi{\'c}.
\newblock On the largest eigenvalue of bicyclic graphs.
\newblock {\em Publ. Inst. Math. (Beograd) (N.S.)}, 46(60):1--6, 1989.

\bibitem{Smith1970}
John~H. Smith.
\newblock Some properties of the spectrum of a graph.
\newblock In {\em Combinatorial {S}tructures and their {A}pplications}, pages
  403--406. Gordon and Breach, New York, 1970.

\bibitem{Stankova1994}
Zvezdelina~E. Stankova.
\newblock Forbidden subsequences.
\newblock {\em Discrete Math.}, 132(1-3):291--316, 1994.

\bibitem{Vatter2010b}
Vincent Vatter.
\newblock Permutation classes of every growth rate above 2.48188.
\newblock {\em Mathematika}, 56(1):182--192, 2010.

\bibitem{Vatter2011}
Vincent Vatter.
\newblock Small permutation classes.
\newblock {\em Proc. Lond. Math. Soc.}, 103(5):879--921, 2011.

\bibitem{WatonThesis}
Stephen~D. Waton.
\newblock {\em On permutation classes defined by token passing networks,
  gridding matrices and pictures: Three flavours of involvement}.
\newblock PhD thesis, University of St Andrews, 2007.

\bibitem{WikiEnumPermClasses}
Wikipedia.
\newblock Enumerations of specific permutation classes.
\newblock
  \href{http://en.wikipedia.org/wiki/Enumerations\_of\_specific\_permutation\_classes}{http:/$\!$/en.wikipedia.org/wiki/Enumera}
  \href{http://en.wikipedia.org/wiki/Enumerations\_of\_specific\_permutation\_classes}{tions\_of\_specific\_permutation\_classes}.

\bibitem{WN2007}
Renee Woo and Arnold Neumaier.
\newblock On graphs whose spectral radius is bounded by {$\frac32\sqrt2$}.
\newblock {\em Graphs Combin.}, 23(6):713--726, 2007.

\bibitem{YFI2010}
J.~Yarkony, C.~Fowlkes, and A.~Ihler.
\newblock Covering trees and lower-bounds on quadratic assignment.
\newblock In {\em IEEE Conference on Computer Vision and Pattern Recognition},
  pages 887--894, 2010.

\end{thebibliography}

\end{document}